\documentclass{amsart}
\usepackage{hyperref}
\usepackage{fancyhdr}
\usepackage{graphicx}
\usepackage{caption}
\hypersetup{
	colorlinks=true,
	anchorcolor=blue,
	linkcolor=blue,
	filecolor=blue,
	urlcolor=blue,
	citecolor=blue,
	bookmarks=true,
	bookmarksopen=true,
	pdfborder=000
}
\usepackage{amsmath}
\usepackage{amsfonts}
\usepackage{amsthm}
\usepackage{amssymb}
\usepackage{enumerate}

\theoremstyle{plain}
\newtheorem{thm}{Theorem}[section]

\newtheorem{cor}[thm]{Corollary}

\newtheorem{lemma}[thm]{Lemma}
\theoremstyle{definition}
\newtheorem{defn}[thm]{Definition}
\theoremstyle{remark}
\newtheorem{rmk}[thm]{Remark}

\makeatletter

\newcommand{\Rmnum}[1]{\expandafter\@slowromancap\romannumeral #1@}
\numberwithin{equation}{section}

\makeatother
\begin{document}
	\title{The growth rate on the volume of $\mathcal{M}_g^{<L(g)}$}
	
	\author{Jinsong Liu}
	\address[J. ~Liu, X. ~Shan]{State Key Laboratory of Mathematical Sciences, Academy of Mathematics and Systems Science, Chinese Academy of Sciences, Beijing 100190, China}
	\address{School of Mathematical Sciences, University of Chinese Academy of Sciences, Beijing 100049, China}
	\email{liujsong@math.ac.cn, shanxu2021@amss.ac.cn}
	
	\author{Xu Shan}
	
	\author{Lang Wang}
	\address[L. ~Wang]{School of Mathematical Sciences, Guizhou Normal University, Guiyang 550025, China}
	\email{wanglang2020@amss.ac.cn}
	
	\author{Yaosong Yang}
	\address[Y. ~Yang]{Beijing International Center for Mathematical Research (BICMR), Beijing 100871, China}
	\email{yangyaosong@amss.ac.cn}
	
	\maketitle
	\begin{abstract}
		Let $\mathcal{M}_g$ be the moduli space of hyperbolic surfaces of genus g endowed with the Weil-Petersson metric. In this paper, we introduce a function $L(g)$ of genus $g$ and call the geodesics whose length less than $L(g)$ short geodesics. We compute the growth rate on the volume of the subset of hyperbolic surfaces with short geodesics. In particular, when $g$ approaches infinity, if $L(g)$ also approaches infinity, then the volume of surfaces characterized by short geodesics is equal to $V_g$ almost surely.
	\end{abstract}

	\section{\noindent{\bf Introduction}}

	Let $\mathbf{L}=(L_1,\dots,L_n)$ be a $n$-type in $\mathbb{R}^n_+$. Denote $\mathcal{T}_{g,n}(\mathbf{L})$ to be the Teich\"uller space consists of hyperbolic structures with genus $g$ and $n$ marked geodesic boundary components of lengths $L_1,\dots,L_n$. If $L_i=0$ for some $i$, it indicates that the $i$-th boundary component is a cusp. If all boundary components are cusps, we will simply denote $\mathcal{T}_{g,n}(\mathbf{L})$ as $\mathcal{T}_{g,n}$. The moduli space $\mathcal{M}_{g,n}(\mathbf{L})$ is the quotient of $\mathcal{T}_{g,n}(\mathbf{L})$ by the action of the mapping class group. We will simply denote $\mathcal{M}_{g,0}$ as $\mathcal{M}_g$.
	
	The length of the shortest geodesic on hyperbolic surfaces is a fundamental quantity that has been extensively studied in recent years. In 1985, Bers \cite{bers1985inequality} proved the existence of a universal constant $l(g)$ such that every closed hyperbolic surface of genus $g$ admits a closed geodesic of length at most $l(g)$. Subsequent work by Buser \cite{buser2010geometry} established the refined estimates for the lengths and counts of short geodesics, exploring their profound implications for the geometry and spectral properties of Riemann surfaces. The primary technique in their work is the area-length method, employed to estimate the length of the shortest geodesics.
	
	On the other hand, the Weil-Petersson metric on Teichum\"uller space has been extensively investigated in recent years, and is related to the mentioned topics above. A pivotal advancement came from Mirzakhani's work \cite{mirzakhani2007simple}, which links Weil-Petersson volumes of moduli spaces to the enumeration of simple closed geodesics, providing powerful new tools for analyzing geodesic distributions. Subsequently, Mirzakhani used the integral formula to estimate the volume of the moduli space of the hyperbolic surfaces with short enough geodesics, and the approximate length of the shortest separating geodesics in \cite{mirzakhani2013growth}. Furthermore, Mirzakhani and Petri \cite{mirzakhani2019lengths} computed the large genus asymptotic behavior of the expected systole length on random surfaces, while Nie, Wu and Xue \cite{nie2023large} investigated the length spectrum of separating closed geodesics.

	Let $\mathcal{M}_g^{<t}$ denote the subset of $\mathcal{M}_g$ consisting of all surfaces $X$ that admit closed geodesics of length less than $t$ and $\mathcal{M}_g^{\geqslant t}$ denote its complement. For a fixed small constant $\varepsilon >0$, Mirakhani \cite{mirzakhani2013growth} demonstrated that the volume of $\mathcal{M}_g^{<\varepsilon}$ satisfies $\lim\limits_{g\rightarrow\infty}\frac{\operatorname{Vol}(\mathcal{M}_g^{<\varepsilon})}{V_g}=O(\varepsilon^2)$, where $V_g$ is the Weil-Petersson volume of $\mathcal{M}_g$.   Concurrently, Mirzakhani and Petri \cite{mirzakhani2019lengths} provided a skecth proof that the probability of $X\in\mathcal{M}_g^{\geqslant L(g)}$ is bounded by $O(L(g)e^{-L(g)})$ for $L(g)<A\log g$. After that,  Lipnowski and Wright \cite{lipnowski2024towards} derived the precise volume of $\mathcal{M}_g^{<\varepsilon}$ using the principle of inclusion-exclusion. In this paper, when $L(g)\approx \log\log g$, we obtain a more precise volume estimate of $M_g^{<L(g)}$ than the one in Mirzakhani and Petri \cite{mirzakhani2019lengths}. For simplicity, $L(g)$ will be abbreviated as $L$ in the rest of the article.
	
	The main result is as following.
	
	\begin{thm}
		\label{Thm1.1}
		For any $\varepsilon>0$, let $L(g)$ be a function of genus $g$ satisfying
		\begin{equation}
			\notag		L(g)\geqslant (1+\varepsilon)\log \log g.
		\end{equation}
		Then we have
		\begin{equation}
			\notag    \operatorname{Prob}_{\mathrm{WP}}^g (X\in\mathcal{M}_g ; \ell_{sys}(X)\leqslant L(g)) =1-O(g^{-\frac{1}{2}+o(1)}),
		\end{equation}
		where we follow the notations in \cite{mirzakhani2013growth} to denote $\operatorname{Prob}_{\mathrm{WP}}^g(\mathcal{A}):=\frac{1}{V_g} \int_{\mathcal{M}_g} \mathbf{1}_{\mathcal{A}} d X$
		and let $\ell_{sys}(X)$ be the length of the shortest closed geodesic on $X$.
		
	\end{thm}
	
	To prove Theorem \ref{Thm1.1}, we partition the moduli space $\mathcal{M}_g$ into $\mathcal{N}_g^L$ (see Definition \ref{Defn3.1}) and its complement. And we first give a similar result on the subset $\mathcal{N}_g^L$.

	\begin{thm}
		\label{Thm1.11}
		For $\delta>0$, let $L(g)$ be a function of genus $g$ satisfying
		\begin{equation}
			\notag		\lim\limits_{g\rightarrow\infty}L(g)=\infty\quad \text{and}\quad L(g)\leqslant \log\left((2-\delta)\log g\right),
		\end{equation}
		then we have
		
		\begin{equation}
			\notag    	\lim\limits_{g\rightarrow\infty} \frac{\operatorname{Vol}(\mathcal{N}_g^{L(g)})}{V_g}=1.
		\end{equation}
		
		In particular, there exists a function $L_0(g)$ such that
		\begin{equation}
			\notag		\frac{\operatorname{Vol}(\mathcal{N}_g^{L_0(g)})}{V_g }=1-O(g^{-\frac{1}{2}+o(1)}).
		\end{equation}
	\end{thm}
	
	To obtain Theorem \ref{Thm1.11}, we first show $\mathcal{N}_g^L$ is a closed set in $\mathcal{M}_g^{<L}$, and thus measurable in Sect. \ref{S3}. And we also need to verify Mirzakhani's integral formula on $\mathcal{N}_{g,n}(\mathbf{L})$ as following, see the details in Sect. \ref{S4}.
	

	\begin{thm}
		\label{Thm1.2}
		Suppose $\gamma_i$ are simple primitive closed curves. For any multi-curve $\gamma=\sum\limits_{i=1}^{k}c_i\gamma_i$, $\ell_{\gamma_i}(X)< L$, the integral of $f_\gamma$ over $\mathcal{N}_{g,n}^L(\mathbf{L})$ with respect to the Weil-Petersson volume form is given by
		\begin{equation}
			\notag \int_{\mathcal{N}_{g,n}^L(\mathbf{L})}f_{\gamma}(X)dX=\frac{2^{-M(\gamma)}}{|Sym(\gamma)|}\int_{x\in\mathbb{R}^k_+}f(|x|)V_{g,n}^L(\Gamma,x,\beta,\mathbf{L})x\cdot dx,
		\end{equation}
		where $\Gamma=(\gamma_1,\dots,\gamma_k)$, $|x|=\sum\limits_{i=1}^{k}c_ix_i$, $x\cdot dx=x_1\dots x_k\cdot dx_1\wedge\dots\wedge dx_k$,\, $\mathbf{L}=(L_1,\dots,L_n),\, L_i\leqslant L$ for $i=1,\dots,n$ and
		\begin{equation}
			\notag M(\gamma):= |\{i|\gamma_i \text{ separates off a one-handle from } S_{g,n}\}|,
		\end{equation}
		and $V_{g,n}^L(\Gamma,x,\beta,\mathbf{L})$ is defined by
		\begin{equation}
			\notag V_{g,n}^L(\Gamma,x,\beta,\mathbf{L}):=\prod_{i=1}^{s}\operatorname{Vol}\left(\mathcal{N}_{g_i,n_i}^L(\ell_{A_i})\right),
		\end{equation}
		where
		\begin{equation}
			\notag S_{g,n}(\gamma)=\bigcup_{i=1}^{s}S_i,
		\end{equation}
		$S_i\simeq S_{g_i,n_i}$, and $A_i=\partial S_i$.
	\end{thm}
	
	To estimate the volume of $\mathcal{N}_{g,n}^L(\ell_{A_i})$ in Theorem \ref{Thm1.2}, we define an important quantity on $X\in\mathcal{M}_g$ (or $\mathcal{M}_{g,n}^L(\mathbf{L})$) that
	\begin{equation}
		\notag\mathcal{L}_1(X):= \min\{\ell_{\gamma}(X)| \gamma=\gamma_1+\dots+\gamma_k \text{ is a simple closed multi-geodesics separating } X \}.
	\end{equation}
	
	Then we prove the following estimate:
	\begin{thm}
		Let $n=n(g)$ and $L_i\leqslant L=L(g)$ for $i=1,\dots,n$, such that
		\begin{equation}
			\notag \varlimsup\limits_{g\rightarrow\infty}\frac{\log n(g)}{\log g}=0,\, \varlimsup\limits_{g\rightarrow\infty}\frac{\log L(g)}{\log g}=0.
		\end{equation}
		Then we have
		\begin{equation}
			\notag \frac{1}{V_{g,n}(\mathbf{L})}\operatorname{Vol}(E_{g,n}^L(\mathbf{L}))\leqslant c_1\frac{\widetilde{L}^2}{g}e^{\frac{L}{2}}+c_2e^{\frac{L}{2}}\frac{\widetilde{L}^{7}}{g^2}
			+c_3e^{2L}\frac{n^{10}}{g^{9}}+c_4e^{2L}\frac{n^{12}}{g^{10}},
		\end{equation}
		where $\widetilde{L}=\max\{L,n\}$, $c_1,c_2,c_3,c_4>0$ are some positive constants and $E_{g,n}^L(\mathbf{L})$ is defined as a subset of $\mathcal{M}_{g,n}(\mathbf{L})$ that
		\begin{equation}
			\notag E_{g,n}^L(\mathbf{L}) :=\{X\in\mathcal{M}_{g,n}(\mathbf{L})| \mathcal{L}_1(X)\leqslant L\}.
		\end{equation}
		
		In particular, if we take $n<c\log g$ and $ L=(2-\varepsilon)\log g$ for any $0<\varepsilon<2$ and $0\leqslant c$, then
		\begin{equation}
			\notag \lim\limits_{g\rightarrow\infty}\frac{1}{V_{g,n}(\mathbf{L})}\operatorname{Vol}(E_{g,n}^{(2-\varepsilon)\log g}(\mathbf{L}))=0.
		\end{equation}
	\end{thm}
	\begin{rmk}
		By this theorem, the volume of the subset $\mathcal{N}_{g,n}^L(\mathbf{L})$ can be derived. For details, see Corollary \ref{C4.2}.
	\end{rmk}
	
	\bigskip
	This paper is organized as follows. We first show the subset $\mathcal{N}_g^{L}$ within $\mathcal{M}_g^{<L}$ is closed, and thus measurable under the Weil-Petersson metric in Sect. \ref{S3}. Then we formulate Mirzakhani's integral formula on $\mathcal{N}_{g,n}^{L}(\mathbf{L})$ in Sect. \ref{S4}. Next in Sect. \ref{S5}, we use the idea of Nie, Wu and Xue \cite{nie2023large} to estimate the volume of the subset $\mathcal{N}_{g,n}^L(\mathbf{L})$. In particular, we show that the total length of a multi-geodesic, which divided the hyperbolic surface in $\mathcal{M}_{g,n}(\mathbf{L})$ into two parts, is also almost greater than $(2-\varepsilon)\log g$ under the Weil-Petersson metric, which we use to estimate the volume of $\mathcal{N}_{g,n}^{L}(\mathbf{L})$ where $n,L=g^{o(1)}$. Combined the properties of $\mathcal{N}_{g,n}^L(\mathbf{L})$ established in previous parts, we estimate the growth rate on the volume of $\mathcal{N}_g^L, \mathcal{M}_g^{<L}$, and thus prove Theorem \ref{Thm1.11} and Theorem \ref{Thm1.1} in Sect. \ref{S6}.\\

	\section{\noindent{\bf The closedness of $\mathcal{N}_g^L$}}
	\label{S3}
	
	\subsection{\noindent{Collar theorem}}

	Let $\alpha$ be a closed curve on a hyperbolic surface $X$, and we denote the length of the curve $\alpha$ as $\ell(\alpha)$.
	We consider a simple closed geodesic $\gamma$ on $X$, according to the collar theorem \cite{keen1974collars}, there exists a collar $\mathcal{C}(\gamma,w(\gamma))$ with width $w(\gamma)$ that
	\begin{equation}
		\notag \mathcal{C}(\gamma,w(\gamma))=\{p\in X| \operatorname{dist}(p,\gamma)< w(\gamma)\}
	\end{equation}
	is isometric to the cylinder $(-w(\gamma),w(\gamma))\times S^1$ with the Riemannian metric $ds^2=d\rho^2+\ell^2(\gamma)\cosh ^2\rho dt^2$ where $w(\gamma)\geqslant\operatorname{arcsinh}\frac{1}{\sinh(\frac{1}{2}\ell(\gamma))}$.
	
	If $\ell(\gamma)<2\operatorname{arcsinh}1$, then the width of its associated collar is at least $\operatorname{arcsinh}1$. Consequently, if two geodesics $\alpha_1$ and $\alpha_2$ satisfy $\ell(\alpha_1),\ell(\alpha_2)< 2\operatorname{arcsinh}1$, they must be disjoint.
	
	Similarly, for a geodesic boundary $\beta$, there also exists a corresponding half collar $\mathcal{C}(\beta,w(\beta))$ with width $w(\beta)$ such that
	\begin{equation}
		\notag  \mathcal{C}(\beta,w(\beta))=\{p\in X|\operatorname{dist}(p,\beta)< w(\beta)\}
	\end{equation}
	is isometric to the cylinder $[0,w(\beta))\times S^1$ with the Riemannian metric $ds^2=d\rho^2+\ell^2(\beta)\cosh ^2\rho dt^2$, where $w(\beta)\geqslant\operatorname{arcsinh}\frac{1}{\sinh(\frac{1}{2}\ell(\gamma))}$.
	
	Furthermore, all collars $\mathcal{C}(\gamma,\frac{1}{2}\operatorname{arcsinh}1)$ and half collars $\mathcal{C}(\beta,\frac{1}{2}\operatorname{arcsinh}1)$ are mutually disjoint, for geodesics $\gamma$ and boundaries $\beta$ with length $\ell(\gamma),\ell(\beta)< 2\operatorname{arcsinh}1$. We now replace $2\operatorname{arcsinh}1$ with $L$.
	
	\subsection{\noindent{The closedness of $\mathcal{N}_g^L$}}
	
	Let $L=L(g)>0$ be a fixed positive number or a positive function depending on the genus $g$. Now we define $\mathcal{N}_g^L$ as a subset of $\mathcal{M}_g^{<L}$.
	\begin{defn}
		\label{Defn3.1}   Let $\mathcal{N}_g^L$ be the subset of $\mathcal{M}_g^{< L}$ such that
		\begin{align}
			\notag \mathcal{N}_g^L:=\{&X\in\mathcal{M}_g^{<L}|\text{all primitive closed geodesics $\gamma$ on $X$ whose length $\ell(\gamma)< L$}\\
			\notag  &\text{have collars with widths $\geqslant \frac{1}{4}L$ and the interior of the collars $\mathcal{C}(\gamma,\frac{1}{4}L)$}\\ \notag&\text{are disjoint}\}.
		\end{align}
	\end{defn}
	For the space of the surfaces with $n$ marked geodesic boundary components of lengths $\mathbf{L}$, $L_i<L$, we give a similar definition as following.
	\begin{defn}
		Let $\mathcal{N}_{g,n}^L(\mathbf{L})$ be a subset of $\mathcal{M}_{g,n}(\mathbf{L})$ that
		\begin{align}
			\notag\mathcal{N}_{g,n}^L(\mathbf{L}):=\{&X\in\mathcal{M}_{g,n}(\mathbf{L})|\text{all primitive closed geodesics $\gamma$ on $X$ whose length $\ell(\gamma)$}\\
			\notag& \text{$< L$ (if $X$ has) have collars with width $\geqslant\frac{1}{4}L$. All the length of} \\
			\notag& \text{boundary geodesics $\beta<L$ and have half-collars with width $\geqslant\frac{1}{4}L$}\\
			\notag& \text{and all the interior of these collars and half-collars are disjoint}\}.
		\end{align}
	\end{defn}
	\begin{rmk}
		By definition, there exists at least one closed short geodesic on $X\in\mathcal{N}_g^L$, however, there might be no closed short geodesic on $X\in\mathcal{N}^L_{g,n}(\mathbf{L})$ except for boundaries. What's more, all primitive short geodesics on $X\in\mathcal{N}_g^L$ or $X\in\mathcal{N}_{g,n}^L(\mathbf{L})$ are simple and disjoint. Therefore, there can be at most $3g-3$ short geodesics on $X\in\mathcal{N}_g^L$ and at most $3g-3+n$ short geodesics on $X\in\mathcal{N}_{g,n}^L(\mathbf{L})$.
	\end{rmk}
	To get the Mirzakhani's integral formula on $\mathcal{N}_g^L$, we need to prove the closedness of $\mathcal{N}_g^L$ within $\mathcal{M}_g^{<L}$, which implies $\mathcal{N}_g^L$ is measurable.
	\begin{thm}
		$\mathcal{N}_g^L$ is a closed subset within $\mathcal{M}_g^{<L}$. It follows that $\mathcal{N}_g^L$ is a Borel measurable set.
	\end{thm}
	\begin{proof}
		We consider a sequence $\{X_n\}$ in $\mathcal{N}_g^L$ and $X_n\rightarrow X\in\mathcal{M}_g^{<L}$. Then there is a sequence of $q_n$-quasi isometries $\{f_n\}$ with $f_n:X_n\rightarrow X$ and $q_n\rightarrow 1$.
		
		For any two geodesics $\gamma_i$ and $\gamma_j$ in $X$ whose lengths $\ell(\gamma_i)$ and $\ell(\gamma_j)$ are shorter than $L$, we claim that $\gamma_i$ and $\gamma_j$ are disjoint. Otherwise, $q_n\rightarrow 1$, there exists $n_0$ that
		\begin{equation}
			\notag q_{n_0}\ell(\gamma_i), q_{n_0}\ell(\gamma_j)< L.
		\end{equation}
		Let $\gamma_i^{n_0}$, $\gamma_j^{n_0}$ on $X_{n_0}$ be the geodesics which are homotopic to $f_{n_0}^{-1}(\gamma_i)$, $f_{n_0}^{-1}(\gamma_j)$, respectively, with lengths satisfying
		\begin{align}
			\notag\ell(\gamma_i^{n_0})&\leqslant \ell(f_{n_0}^{-1}(\gamma_i))\leqslant q_{n_0}\ell(\gamma_i)< L,\\
			\notag\ell(\gamma_j^{n_0})&\leqslant \ell(f_{n_0}^{-1}(\gamma_j))\leqslant q_{n_0}\ell(\gamma_j)< L.
		\end{align}
		
		Hence, the geodesics $\gamma_i^{n_0}$ and $\gamma_j^{n_0}$ are two geodesics whose lengths are shorter than $L$ on $X_{n_0}$. Since $f_{n_0}^{-1}(\gamma_i)$ and $f_{n_0}^{-1}(\gamma_j)$ intersect transversely, by \cite[Theorem 1.6.7]{buser2010geometry}, $\gamma_i^{n_0}$ and $\gamma_j^{n_0}$ intersect transversely, which is a contradiction.\\
		
		Due to the discreteness of the length spectrum of a hyperbolic surface, there are only finitely many primitive closed geodesics on $X\in\mathcal{M}_g$ with length less than $L$. We denote these simple geodesics by $\gamma_1,\dots,\gamma_m$.
		
		We already know that these geodesics are pairwise disjoint and simple. Furthermore, by collar theorem, there exists additional simple closed geodesics $\gamma_{m+1},\dots,\gamma_{3g-3}$ that, together with $\gamma_1,\dots,\gamma_m$, decompose $X$ into pairs of pants.
		
		Next we consider the map from the Teich\"uller space $\mathcal{T}_g$ to the moduli space by the quotient map $q:\mathcal{T}_g\rightarrow\mathcal{M}_g$, where
		\begin{equation}
			\notag  \mathcal{M}_g=\mathcal{T}_g/\text{Mod}_g,
		\end{equation}
		and $\text{Mod}_g$ is the modular group.\\
		
		Let $G$ denote the marked cubic graph corresponding to the pants decomposition. By \cite[Corollary 6.2.8]{buser2010geometry} and \cite[Lemma 6.3.4]{buser2010geometry}, if follows that there exists an open neighborhood $U$ of $X$ in $\mathcal{M}_g$ and a real analytic mapping $q^{-1}:U\rightarrow\mathcal{R}^{6g-g}$ and derive the Fenchel-Nielsen coordinates $\omega_X$ for $\widetilde{X}=q^{-1}(X)$ that
		\begin{equation}
			\notag  \omega_X=(\ell_{\gamma_1},\dots,\ell_{\gamma_{3g-3}},\tau_1,\dots,\tau_{3g-3}),
		\end{equation}
		where $\tau_k$ is the twisting parameter on $\gamma_k$ of $X$.
		
		For $X_n\in \mathcal{N}_g^L\cap U$, there exists $\widetilde{X_n}\in\mathcal{T}_G$ such that $q(\widetilde{X_n})=X_n$, corresponding the Fenchel-Nielsen parameter $\omega_{X_n}$.
		
		Given any hyperbolic surface $Y\in\mathcal{M}_g^{<L}$ and a closed geodesic $\gamma$ with length $< L$ on $Y$, we define
		\begin{equation}
			\notag  w_{\gamma}=\sup\{r>0|\mathcal{C}(\gamma,r)\text{ is a collar}\},
		\end{equation}
		and $w:\mathcal{M}_g\rightarrow\mathbb{R}_+$ that
		\begin{equation}
			\notag w(Y)=\min\{\min_{\ell(\gamma)<L}w_{\gamma},\min_{\ell(\gamma_i),\ell(\gamma_j)<L}\frac{1}{2}\operatorname{dist}(\gamma_i,\gamma_j)\}.
		\end{equation}
		
		By definition, we have
		\begin{equation}
			\notag  \mathcal{N}_g^L=\{X\in\mathcal{M}_g|w(X)\geqslant\frac{1}{4}L\}.
		\end{equation}
		According to the definitions of the mapping in \cite[Lemma 3.2.5]{buser2010geometry}, we can see that there exists a $k_n$-quasi-isometry $\sigma[X_n,X]$ between $X_n$ and $X$. Moreover, $\sigma[X_n,X]$ preserves the $3g-3$ marked closed geodesics. Additionally, for any geodesic $\gamma$ on $X$ and $\gamma^{(n)}=\sigma[X_n,X]^{-1}(\gamma)$ we have
		\begin{equation}
			\notag  \mathcal{C}(\gamma,\frac{1}{k_n}w(X_n))\subseteq\sigma[X_n,X]\left(\mathcal{C}(\gamma^{(n)},w(X_n))\right)\subseteq\mathcal{C}(\gamma,k_nw(X_n))).
		\end{equation}
		
		Since $\mathcal{C}(\gamma^{(n)},w(X_n))$ is an embedding in $X_n$ and $\sigma[X_n,X]$ is a homeomorphism, it follows that $\mathcal{C}(\gamma,\frac{1}{k_n}w(X_n))$ is a collar in $X$ which means $w_\gamma\geqslant \frac{1}{k_n}w(X_n)$.
		By \cite[Theorem 6.4.2]{buser2010geometry}, for a sequence of surfaces $X_n\rightarrow X$, the Fenchel-Nielsen parameter $\omega_{X_n}\rightarrow \omega_X$. Also \cite[Lemma 3.2.6]{buser2010geometry} and \cite[Lemma 3.3.8]{buser2010geometry} demonstrate that $k_n\rightarrow 1$.
		
		Hence we can see that
		\begin{align}
			\notag w(X)&=\min\{\min_{1<i<m}w_{\gamma_i},\min_{1\leqslant i<j\leqslant m}\frac{1}{2}\operatorname{dist}(\gamma_i,\gamma_j)\}\\
			\notag&\geqslant\min\{\min_{1<i<m}\frac{1}{k_n}w(X_n),\min_{1\leqslant i<j\leqslant m}\frac{1}{2k_n}\operatorname{dist}(\gamma_i,\gamma_j)\}\\
			\notag&\geqslant \frac{1}{k_n}w(X_n).
		\end{align}
		
		Now let $n\rightarrow \infty$, then $k_n\rightarrow 1$, and $w(X_n)\geqslant \frac{1}{4}L$ for all $n$, we get $w(X)\geqslant \frac{1}{4}L$ yielding that $X\in\mathcal{N}_g^{L}$. Therefore, we conclude that $\mathcal{N}_g^L$ is a closed subset within $\mathcal{M}_{g}^{<L}$. Using the same method, we can also see that $\mathcal{N}_{g,n}^L(\mathbf{L})$ is a closed set.
	\end{proof}

	\section{\noindent{\bf Mirzakhani's integral formula on $\mathcal{N}_{g,n}^L(\mathbf{L})$}}
	\label{S4}
	For a hyperbolic surface $X\in\mathcal{M}_{g,n}(\mathbf{L})$ and a simple closed curve $\alpha$ on $X$, we denote $[\alpha]$ by the homotopy class of $\alpha$ and $\ell_{\alpha}(X)$ by the length of the geodesic in $[\alpha]$.
	Suppose $\gamma_i$ are simple primitive closed curves, $\gamma=\sum\limits_{i=1}^{k}c_i\gamma_i$ is a multi-curve, and we denote the length of $\gamma$ as
	\begin{equation}
		\notag	\ell_{\gamma}(X)=\sum_{i=1}^{k}c_i\ell_{\gamma_i}(X).
	\end{equation}
	Let $f:\mathbb{R}_+\rightarrow\mathbb{R}_+$ be a continuous function, we define a function of multi-curves	 $f_{\gamma}:\mathcal{M}_{g,n}(\mathbf{L})\rightarrow\mathbb{R}_+$ that
	\begin{equation}
		\notag	f_{\gamma}(X):=\sum_{[\alpha]\in Mod\cdot [\gamma]}f\left(\ell_{\alpha}(X)\right).
	\end{equation}
	Let $S_{g,n}(\gamma)$ denote the surface $S_{g,n}$ cutting along the $\gamma$.
	
	In this section, we prove Theorem \ref{Thm1.2}. First we restate it as following.
	\begin{thm}
		\label{Thm4.1}
		For any multi-curve $\gamma=\sum\limits_{i=1}^{k}c_i\gamma_i$, $\ell_{\gamma_i}(X)< L$, the integral of $f_\gamma$ over $\mathcal{N}_{g,n}^L(\mathbf{L})$ with respect to the Weil-Petersson volume form is given by
		\begin{equation}
			\label{T2} \int_{\mathcal{N}_{g,n}^L(\mathbf{L})}f_{\gamma}(X)dX=\frac{2^{-M(\gamma)}}{|Sym(\gamma)|}\int_{x\in\mathbb{R}^k_+}f(|x|)V_{g,n}^L(\Gamma,x,\beta,\mathbf{L})x\cdot dx,
		\end{equation}
		where $\Gamma=(\gamma_1,\dots,\gamma_k)$, $|x|=\sum\limits_{i=1}^{k}c_ix_i$, $x\cdot dx=x_1\dots x_k\cdot dx_1\wedge\dots\wedge dx_k$, $\mathbf{L}=(L_1,\dots,L_n),\, L_i< L$ for $i=1,\dots,n$ and
		\begin{equation}
			\notag M(\gamma):=|\{i|\gamma_i \text{ separates off a one-handle from } S_{g,n}\}|
		\end{equation}
		and $V_{g,n}^L(\Gamma,x,\beta,\mathbf{L})$ is defined by
		\begin{equation}
			\notag V_{g,n}^L(\Gamma,x,\beta,\mathbf{L}):=\prod_{i=1}^{s}\operatorname{Vol}\left(\mathcal{N}_{g_i,n_i}^L(\ell_{A_i})\right),
		\end{equation}
		where
		\begin{equation}
			\notag S_{g,n}(\gamma)=\bigcup_{i=1}^{s}S_i,
		\end{equation}
		$S_i\simeq S_{g_i,n_i}$, and $A_i=\partial S_i$.
	\end{thm}
	
	To prove Theorem \ref{Thm4.1}, we recall the Mirzakani's integral formula on $\mathcal{M}_{g,n}(\mathbf{L})$.
	
	\begin{lemma}\cite[Theorem 7.1]{mirzakhani2007simple}
		\label{Lem4.2}
		For any multi-curve $\gamma=\sum_{i=1}^{k}c_i\gamma_i$, the integral of $f_{\gamma}$ over $\mathcal{M}_{g,n}(\mathbf{L})$ with respect to the Weil-Petersson volume form is given by
		\begin{equation}
			\notag\int_{\mathcal{M}_{g,n}(\mathbf{L})}f_{\gamma}(X)dX=\frac{2^{-M(\gamma)}}{|Sym(\gamma)|}\int_{x\in\mathbb{R}^k_+}f(|x|)V_{g,n}(\Gamma,x,\beta,\mathbf{L})x\cdot dx,
		\end{equation}
		where $\Gamma=(\gamma_1,\dots,\gamma_k)$, $|x|=\sum_{i=1}^{k}c_ix_i$, $x\cdot dx=x_1\dots x_k\cdot dx_1\wedge\dots\wedge dx_k$ and
		\begin{equation}
			\notag	M(\gamma):=|\{i|\gamma_i \text{ separates off a one-handle from } S_{g,n}\}|,
		\end{equation}
		and $V_{g,n}(\Gamma,x,\beta,\mathbf{L})$ is defined by
		\begin{equation}
			\notag V_{g,n}(\Gamma,x,\beta,\mathbf{L}):=\operatorname{Vol}\left(\mathcal{M}\left(S_{g,n}(\gamma),\ell_{\Gamma}=x,\ell_{\beta}=\mathbf{L}\right)\right)=\prod_{i=1}^{s}V_{g_i,n_i}(\ell_{A_i}).
		\end{equation}
	\end{lemma}
	
	Now we sketch how we replace $\mathcal{M}_{g,n}(\mathbf{L})$ by $\mathcal{N}_{g,n}^L(\mathbf{L})$ in Lemma \ref{Lem4.2}. For an element $h\in Mod_{g,n}$, it acts on $\Gamma$ by
	\begin{equation}
		\notag h\cdot\Gamma=(h\cdot\gamma_1,\dots,h\cdot\gamma_k).
	\end{equation}
	Let $\mathcal{O}_{\Gamma}$ be the set of homotopy classes of elements of $Mod\cdot \Gamma$, and $\mathcal{M}_{g,n}(\mathbf{L})^{\Gamma}$ be the set
	\begin{equation}
		\notag \mathcal{M}_{g,n}(\mathbf{L})^{\Gamma}=\{(X,\eta)|X\in\mathcal{M}_{g,n}(\mathbf{L})^{\Gamma},\eta=(\eta_1,\dots,\eta_k)\in\mathcal{O}_{\Gamma}\}.
	\end{equation}
	Then
	\begin{equation}
		\label{Ngamma} \mathcal{M}_{g,n}(\mathbf{L})^{\Gamma}=\mathcal{T}_{g,n}(\mathbf{L})/G_{\Gamma},
	\end{equation}
	where
	\begin{equation}
		\notag G_{\Gamma}=\bigcap_{i=1}^{k}Stab(\gamma_i)\subseteq Mod(S_{g,n}).
	\end{equation}
	
	Now we define the covering map $\pi^{\Gamma}:\mathcal{M}_{g,n}(\mathbf{L})^{\Gamma}\rightarrow \mathcal{M}_{g,n}(\mathbf{L})$
	\begin{equation}
		\notag \pi^{\Gamma}(X,\eta) :=X.
	\end{equation}
	For $X\in\mathcal{N}_{g,n}^L(\mathbf{L})\subseteq\mathcal{M}_{g,n}(\mathbf{L})$ and $\Gamma=(\gamma_1,\dots,\gamma_k)$, we define
	\begin{equation}
		\notag (\mathcal{N}_{g,n}^L(\mathbf{L}))^{\Gamma}:=(\pi^{\Gamma})^{-1}\mathcal{N}_{g,n}^L(\mathbf{L})
	\end{equation}
	as a subset of $\mathcal{M}_{g,n}(\mathbf{L})^{\Gamma}$. Due to (\ref{Ngamma}), there is a covering map between $\mathcal{T}_{g,n}(\mathbf{L})$ and $\mathcal{M}_{g,n}(\mathbf{L})^{\Gamma}$, we define a subset of $\mathcal{T}_{g,n}(\mathbf{L})$ as $\mathcal{T}_{g,n}^L(\mathbf{L})$ such that there exists a covering map between $\mathcal{T}_{g,n}^L(\mathbf{L})$ and $(\mathcal{N}_{g,n}^L(\mathbf{L}))^{\Gamma}$ satisfying
	\begin{equation}
		\notag (\mathcal{N}_{g,n}^L(\mathbf{L}))^{\Gamma}=\mathcal{T}_{g,n}^L(\mathbf{L})/G_{\Gamma}\subseteq\mathcal{T}_{g,n}(\mathbf{L})/G_{\Gamma}.
	\end{equation}
	Next we consider twisting on $\Gamma$. Let $\ell_{\Gamma}:\mathcal{T}_{g,n}(\mathbf{L})\rightarrow \mathbb{R}_+^k$
	\begin{equation}
		\notag \ell_{\Gamma}(X)=(\ell_{\gamma_1}(X),\dots,\ell_{\gamma_k}(X)).
	\end{equation}
	We restrict $\ell_{\Gamma}$ to the subset $\mathcal{T}_{g,n}^L(\mathbf{L})$ as $\ell^{\Gamma}|_{\mathcal{T}_{g,n}^L(\mathbf{L})}$.
	
	Let $a=(a_1,\dots,a_k)\in\mathbb{R}_+^k$ be a $k$-type such that $a_i< L$ for all $i=1,\dots,k$.
	The Weil-Petersson volume form induces a natural measure on the preimage under the essential projection map $(\ell^{\Gamma}|_{\mathcal{T}_{g,n}^L(\mathbf{L})})^{-1}(a)\subseteq \mathcal{T}_{g,n}^L(\mathbf{L})\subseteq\mathcal{T}_{g,n}(\mathbf{L})$. By the assumption of $\mathcal{N}_{g,n}^L(\mathbf{L})$, each $\gamma_i$ has a collar with width $\geqslant\frac{1}{4}L$, which implies that $\gamma_i\cap\mathcal{C}(\gamma_j,\frac{1}{4}L)=\varnothing$ for $i\neq j$.
	
	Let $\phi_{\gamma_i}^{t}:\mathcal{T}_{g,n}^L(\mathbf{L})\rightarrow\mathcal{T}_{g,n}^L(\mathbf{L})$ be twisting along $\gamma_i$ with length $t$. Then we have
	\begin{equation}
		\notag \phi_{\gamma_i}^{t}(\mathcal{C}(\gamma_j,\frac{1}{4}L))=\mathcal{C}(\gamma_j,\frac{1}{4}L)
	\end{equation}
	for all $i,j=1,\dots k$.
	This implies the twisting map
	\begin{equation}
		\notag \phi_{\gamma}^{(t_1,\dots,t_k)}:=\prod_{i=1}^{k}\phi_{\gamma_i}^{t_i}
	\end{equation}
	is a bijection from $\mathcal{T}_{g,n}^L(\mathbf{L})$ to $\mathcal{T}_{g,n}^L(\mathbf{L})$, which preserves the Weil Petersson symplectic form. It means that for fixed $\mathbf{L}$, the volume of  $\mathcal{N}_{g,n}^L(\mathbf{L})$ depends only on $\ell_{\Gamma}(X)$. This is precisely why we construct the subset $\mathcal{N}_{g,n}^L(\mathbf{L})$ by controlling the widths of the collars around short geodesics.
	From the above discussion, the length function $\ell_{\Gamma}|_{\mathcal{T}_{g,n}^L(\mathbf{L})}$ descends to a function $\mathcal{L}_{\Gamma}|_{(\mathcal{N}_{g,n}^L(\mathbf{L}))^{\Gamma}}$ on $(\mathcal{N}_{g,n}^L(\mathbf{L}))^{\Gamma}$ as
	\begin{equation}
		\notag \mathcal{L}_{\Gamma}|_{(\mathcal{N}_{g,n}^L(\mathbf{L}))^{\Gamma}}\left(X,(\eta_1,\dots,\eta_k)\right)=(\ell_{\eta_1}(X),\dots,\ell_{\eta_k}(X)).
	\end{equation}
	Then the map $\pi=\pi^{\Gamma}\circ\mathcal{L}_{\Gamma}^{-1}:\mathbb{R}_+^k\rightarrow\mathcal{M}_{g,n}(\mathbf{L})$ restricts to a covering map $\pi|_{U}:U\rightarrow\mathcal{N}_{g,n}^L(\mathbf{L})$, where $U=\ell_{\Gamma}(\mathcal{T}_{g,n}^L(\mathbf{L}))\subset\mathbb{R}_+^k$.\\
	
	Now we begin to prove Theorem \ref{Thm4.1} and first recall
	\begin{lemma}\cite[Lemma 7.3]{mirzakhani2007simple}
		For any function $F:\mathbb{R}_+^k\rightarrow\mathbb{R}_+$ and $\Gamma=(\gamma_1,\dots,\gamma_k)$, define $F_{\Gamma}:\mathcal{M}_{g,n}(\mathbf{L})^{\Gamma}\rightarrow\mathbb{R}$ by
		\begin{equation}
			\notag  F_{\Gamma}(Y)=F(\mathcal{L}_{\Gamma}(Y)).
		\end{equation}
		Then the integral of $F_{\Gamma}$ over $\mathcal{M}_{g,n}(\mathbf{L}^{\Gamma})$ is given by
		\begin{equation}
			\notag  \int_{\mathcal{M}_{g,n}(\mathbf{L})^{\Gamma}}F_{\Gamma}(Y)dY=2^{-M(\Gamma)}\int_{x\in\mathbb{R}^+_k}F(x)\operatorname{Vol}(\mathcal{M}(S_{g,n}(\gamma),\ell_{\beta}=\mathbf{L},\ell_{\Gamma}=x))xdx,
		\end{equation}
		where $x=(x_1,\dots,x_k)$, and $xdx=x_1\dots x_kdx_1\wedge\dots\wedge x_k$.
	\end{lemma}
	Using this lemma, we take $F(x_1,\dots,x_k)=f(\sum\limits_{i=1}^{k}c_ix_i)\cdot1_{\mathcal{N}_{g,n}^L(\Gamma,\ell_{\Gamma},\beta,\mathbf{L})}\circ\pi(x)$, then
	\begin{align}
		\notag	\int_{\left(\mathcal{N}_{g,n}(\mathbf{L})\right)^{\Gamma}}F_{\Gamma}(Y)dY=&\int_{\mathcal{M}_{g,n}(\mathbf{L})^{\Gamma}}F_{\Gamma}(Y)dY\\
		\notag=&2^{-M(\Gamma)}\int_{x\in\mathbb{R}_+^k}f(|x|)\cdot1_{\mathcal{N}_{g,n}(\Gamma,\ell_{\Gamma},\beta,\mathbf{L})}\circ\mathcal{L}^{-1}_{\Gamma}(x)\\
		\notag   &\operatorname{Vol}(\mathcal{M}(S_{g,n}(\gamma),\ell_{\beta}=\mathbf{L},\ell_{\Gamma}=x))xdx.
	\end{align}
	
	If $X\in\mathcal{N}_{g,n}^L(\mathbf{L})$ such that $\ell_{\Gamma}=x$, then  we cut $X$ along $\gamma_i$ for $i=1,\dots,k$ and get
	\begin{equation}
		\notag	X-\bigcup_{i=1}^{k}\gamma_i \simeq \bigcup_{i=1}^{s}S_{g_i,n_i}.
	\end{equation}
	
	We can see for $x_i< L$, $\mathcal{C}(\gamma_i,\frac{1}{4}L)\cap\mathcal{C}(\delta,\frac{1}{4}L)=\varnothing$ for all $i=1,\dots,k$ and $\delta$ a closed geodesic on $X$ with length less than $L$.
	Hence $S_{g_i,n_i}\in \mathcal{N}_{g_i,n_i}^L(\ell_{A_i})$ where $A_i=\partial S_{g_i,n_i}$.
	Thus
	\begin{equation}
		\notag 1_{\mathcal{N}_{g,n}(\Gamma,\ell_{\Gamma},\beta,\mathbf{L})}\circ\mathcal{L}^{-1}_{\Gamma}(x)\operatorname{Vol}(\mathcal{M}(S_{g,n}(\gamma),\ell_{\beta}=\mathbf{L},\ell_{\Gamma}=x))=\prod_{i=1}^{s}\operatorname{Vol}\left(\mathcal{N}_{g_i,n_i}^L(\ell_{A_i})\right).
	\end{equation}
	
	Finally, we consider the map $\pi^{\Gamma}:\left(\mathcal{N}_{g,n}(\mathbf{L})\right)^{\Gamma}\rightarrow\mathcal{N}_{g,n}(\mathbf{L})$, then
	\begin{equation}
		\notag	f_{\gamma}\circ \pi^{\Gamma}= F_{\Gamma}.
	\end{equation}
	Combining with
	\begin{equation}
		\notag\int_{\left(\mathcal{N}_{g,n}(\mathbf{L})\right)^{\Gamma}}F_{\Gamma}(Y)dY=2^{-M(\gamma)}\int_{x\in\mathbb{R}_+^k}f(|x|)\operatorname{Vol}(\mathcal{N}_{g,n}^L(\gamma,x,\ell_{\beta},\mathbf{L}))xdx
	\end{equation}
	and
	\begin{equation}
		\notag \int_{\left(\mathcal{N}_{g,n}(\mathbf{L})\right)^{\Gamma}}F_{\Gamma}(Y)dY=|Sym(\gamma)|\int_{\mathcal{N}_{g,n}^L(\mathbf{L})}f_{\gamma}(X)dX,
	\end{equation}
	we thus complete the Mirzakhani's integral formula on $\mathcal{N}_{g,n}^L(\mathbf{L})$.

	\section{\noindent{{\bf The volume of $\mathcal{N}_{g.n}^{L}(\mathbf{L})$}}}
	\label{S5}
	In this section, we aim to determine the length of the shortest separating closed multi-geodesics on $\mathcal{M}_{g,n}(\mathbf{L})$, where $n=g^{o(1)}$ and $\sum\limits_{i=1}^{n}L_i=y=g^{o(1)}$. Subsequently, we will compare the volume of $\mathcal{N}_{g,n}^L(\mathbf{L})$ with $V_{g,n}(\mathbf{L})$. The main theorem is as following.
	
	\begin{thm}
		\label{Thm5.1}
		Let $n=n(g)$ and $L_i< L=L(g)$ for $i=1,\dots,n$, such that
		\begin{equation}
			\notag \varlimsup\limits_{g\rightarrow\infty}\frac{\log n(g)}{\log g}=0,\, \varlimsup\limits_{g\rightarrow\infty}\frac{\log L(g)}{\log g}=0.
		\end{equation}
		Then we have
		\begin{equation}
			\notag \frac{1}{V_{g,n}(\mathbf{L})}\operatorname{Vol}(E_{g,n}^L(\mathbf{L}))\leqslant c_1\frac{\widetilde{L}^2}{g}e^{\frac{L}{2}}+c_2e^{\frac{L}{2}}\frac{\widetilde{L}^{7}}{g^2}
			+c_3e^{2L}\frac{n^{10}}{g^{9}}+c_4e^{2L}\frac{n^{12}}{g^{10}},
		\end{equation}
		where $\widetilde{L}=\max\{L,n\}$, $c_1,c_2,c_3,c_4>0$ are some positive constants, and $E_{g,n}^L(\mathbf{L})$ is defined as a subset of $\mathcal{M}_{g,n}(\mathbf{L})$  such that
		\begin{equation}
			\notag E_{g,n}^L(\mathbf{L}):=\{X\in\mathcal{M}_{g,n}(\mathbf{L});\mathcal{L}_1(X)\leqslant L\}.
		\end{equation}
		
		In particular, taking $n\leqslant c\log g$ and $L=(2-\epsilon)\log g$ for any $0<\varepsilon<2$ and $c \geqslant 0$, then we can get
		\begin{equation}
			\notag \lim\limits_{g\rightarrow\infty}\frac{1}{V_{g,n}(\mathbf{L})}\operatorname{Vol}(E_{g,n}^{(2-\varepsilon)\log g}(\mathbf{L}))=0.
		\end{equation}
	\end{thm}
	
	\begin{proof}
		The proof of this theorem adopts a similar approach presented in \cite{nie2023large}. Specifically, we apply Mirzakhani's integral formula to the moduli space $\mathcal{M}_{g,n}(\mathbf{L})$. Assume $\mathcal{L}_1(X)\leqslant L$ and $\gamma=\gamma_1+\dots+\gamma_k$ satisfies $\ell_{\gamma}(X)=\mathcal{L}_1(X)$, where $\gamma_1,\dots,\gamma_k$ are simple closed geodesics partitioning $X$ into $S_{g_0,k+n_0}\cup S_{g-g_0-k+1,k+n-n_0}$ for some $(g_0,k+n_0)$ with $|\chi(S_{g_0,k+n_0})|\leqslant\frac{1}{2}|\chi(S_{g,n})|=g+\frac{n}{2}-1$.
		
		We define $N_{g_0,k,n_0}(X,L)$ as the number of $k$ closed geodesics $\gamma_1,\dots,\gamma_k$ that separate $X$ into $S_{g_0,k+n_0}\cup S_{g-g_0-k+1,k+n-n_0}$ and satisfy $\ell_{\gamma}(X)=\sum\limits_{i=1}^{k}\ell_{\gamma_i}(X)\leqslant L$. Furthermore, we define $N_{g_0,m}(X,L):=\sum\limits_{k=1}^{m}N_{g_0,k,m-k}(X,L)$. Then
		\begin{align}
			\notag	&\frac{\operatorname{Vol}(\{X\in\mathcal{M}_{g,n}(\mathbf{L});\mathcal{L}_1(X)\leqslant L\})}{V_{g,n}(\mathbf{L})}\\
			\notag	\leqslant&\frac{1}{V_{g,n}(\mathbf{L})}\sum_{(g_0,k,n_0)}\operatorname{Vol}\left(X\in \mathcal{M}_{g,n}(\mathbf{L});N_{g_0,k,n_0}(X,L)\geqslant 1\right)\\
			\notag	\leqslant& \sum_{(g_0,k,n_0)}\frac{1}{V_{g,n}(\mathbf{L})}\int_{\mathcal{M}_{g,n}(\mathbf{L})} N_{g_0,k,n_0}(X,L)dX,
		\end{align}
		where the sum over all $1\leqslant 2g_0-2+k+n_0\leqslant [g-1+\frac{n}{2}]$.
		
		By Mirzakhani integral formula \cite[Theorem 7.1]{mirzakhani2007simple}, we have
		\begin{align}
			\notag\int_{\mathcal{M}_{g,n}(\mathbf{L})} N_{g_0,k,n_0}(X,L)dX =&\sum_{i_1<\dots< i_{n_0}}\frac{2^{-M(\gamma)}}{|Sym(\gamma)|}\int_{\mathbb{R}^k_{+}} 1_{[0,L]}(x_1+\dots+x_k)\\
			\notag&\times V_{g_0,k+n_0}(x_1,\dots,x_k,L_{i_1},\dots,L_{i_{n_0}})\\
			\notag&\times V_{g-g_0-k+1,k+n-n_0}(x_1,\dots,x_k,L_{i_{n_0+1}},\dots,L_{i_n})\\
			\notag \leqslant&\sum_{i_1<\dots< i_{n_0}}\frac{2^{-M(\gamma)}}{k!}\int_{\mathbb{R}^k_{+}} 1_{[0,L]}(x_1+\dots+x_k)\\
			\notag&\times V_{g_0,k+n_0}(x_1,\dots,x_k,L_{i_1},\dots,L_{i_{n_0}})\\
			\notag&\times V_{g-g_0-k+1,k+n-n_0}(x_1,\dots,x_k,L_{i_{n_0+1}},\dots,L_{i_n}),
		\end{align}
		where $i_1<\dots<i_{n_0}$ and $i_{n_0+1}<\dots<i_n$, $i_1\dots i_n$ is a permutation of $\{1,\dots,n\}$ and $M(\gamma)$ is defined by
		\begin{equation}
			\notag	M(\gamma) :=
			\begin{cases}
				1, & \text{if } (g_0,k+n_0)=(1,1) \\
				0, & \text{otherwise}.
			\end{cases}
		\end{equation}
		
		Let $P\geqslant 2$ be a fixed positive integer. For the hyperbolic surface $S_{g_0,k+n_0}$, the Euler characterstic which is denoted as $\chi(S_{g_0,k+n_0})$ has three cases:\\
		
		$(\romannumeral1)$ If $|\chi(S_{g_0,k+n_0})|=1$, then $(g_0,k+n_0)=(1,1)$ or $(0,3)$.\\
		
		\textbf{Case a}: If $(g_0,k+n_0)=(1,1)$. From Lemma \ref{Lem4.2}, we obtain that
		\begin{align}
			\notag\int_{\mathcal{M}_{g,n}(\mathbf{L})} N_{1,1}(X,L)dX&=\int_{\mathcal{M}_{g,n}(\mathbf{L})} N_{1,1,0}(X,L)dX\\
			\notag&=\frac{1}{2}\int_{0}^{L}xV_{1,1}(x)V_{g-1,1+n}(x,L_1,\dots,L_n)dx.
		\end{align}
		By table in \cite{mirzakhani2007simple}, \cite[Lemma 3.2]{mirzakhani2013growth} and \cite[Theorem 3.5]{mirzakhani2013growth}, we have
		\begin{align}
			\notag V_{1,1}(x)&=\frac{1}{24}(x^2+4\pi^2)\\
			\notag V_{g,n}(x_1,\dots,x_n)&\leqslant V_{g,n}\prod_{i=1}^{n}\frac{\sinh\frac{x_i}{2}}{\frac{x_i}{2}}\\
			\notag V_{g-1,n+1}&=O\left(\frac{1}{g}\right)V_{g,n},
		\end{align}
		which implies that
		
		\begin{align}
			\label{A1}     &\int_{\mathcal{M}_{g,n}(\mathbf{L})} N_{1,1}(X,L)dX\\
			\notag \leqslant& \int_{0}^{L}\frac{1}{48}(x^2+4\pi^2)x\frac{\sinh\frac{x}{2}}{\frac{x}{2}}\prod_{i=1}^{n}\frac{\sinh \frac{L_i}{2}}{\frac{L_i}{2}}V_{g-1,1+n}dx\\
			\notag =& O\left(L^2e^{\frac{L}{2}}\right)\prod_{i=1}^{n}\frac{\sinh \frac{L_i}{2}}{\frac{L_i}{2}}V_{g-1,1+n}\\
			\notag =& O\left(\frac{L^2}{g}e^{\frac{L}{2}}\right)\prod_{i=1}^{n}\frac{\sinh \frac{L_i}{2}}{\frac{L_i}{2}}V_{g,n}.
		\end{align}
		
		\textbf{Case b}: If $(g_0,k+n_0)=(0,3)$. It follows that
		\begin{align}
			\label{B2}	&\int_{\mathcal{M}_{g,n}(\mathbf{L})}N_{0,3}(X,L)dX\\
			\notag	=& \sum_{k=1}^{3}\int_{\mathcal{M}_{g,n}(\mathbf{L})}N_{0,k,3-k}(X,L)dX\\
			\notag	\leqslant& \sum_{i_1<i_2}\int_{0}^{L}xV_{0,3}(x,L_{i_1},L_{i_2})V_{g,n-1}(x,L_{i_3},\dots,L_{i_n})dx\\
			\notag	&+\frac{1}{2!}\sum_{i=1}^{n}\int_{0\leqslant x+y\leqslant L}xyV_{0,3}(x,y,L_i)V_{g-1,n+1}(x,y,L_1,\dots,\widehat{L_i}\dots,L_n)dxdy\\
			\notag	&+\frac{1}{3!}\int_{0\leqslant x+y+z\leqslant L}xyzV_{0,3}(x,y,z)V_{g-2,n+3}(x,y,z,L_1,\dots,L_n)dxdydz\\
			\notag  =: &I_1+I_2+I_3.
		\end{align}
		
		Again, by the table in \cite{mirzakhani2007simple} and \cite[Lemma 3.2]{mirzakhani2013growth}:
		\begin{align}
			\notag V_{0,3}(x,y,z)&=1,\\
			\notag V_{g,n}(x_1,\dots,x_n)&\leqslant V_{g,n}\prod_{i=1}^{n}\frac{\sinh\frac{x_i}{2}}{\frac{x_i}{2}},
		\end{align}
		we can get
		\begin{align}
			\label{B3} I_1&\leqslant\sum_{i_1<i_2}\int_{0}^{L}xV_{g,n-1}\frac{\prod_{i=1}^{n}\frac{\sinh\frac{L_i}{2}}{\frac{L_i}{2}}}{\frac{\sinh \frac{L_{i_1}}{2}}{\frac{L_{i_1}}{2}}\frac{\sinh\frac{L_{i_2}}{2}}{\frac{L_{i_2}}{2}}}\frac{\sinh\frac{x}{2}}{\frac{x}{2}}dx\\
			\notag &=\prod_{i=1}^{n}\frac{\sinh\frac{L_i}{2}}{\frac{L_i}{2}}V_{g,n-1}O\left(n^2e^{\frac{L}{2}}\right).
		\end{align}
		
		Similarly, we estimate $I_2$ and $I_3$ as following,
		\begin{align}
			\label{B4}
			I_2&\leqslant\frac{1}{2!}\sum_{i=1}^{n}\int_{0\leqslant x+y\leqslant L}xyV_{g-1,n+1}\frac{\prod_{j=1}^{n}\frac{\sinh \frac{L_j}{2}}{\frac{L_j}{2}}}{\frac{\sinh \frac{L_i}{2}}{\frac{L_i}{2}}}\frac{\sinh \frac{x}{2}}{\frac{x}{2}}\frac{\sinh\frac{y}{2}}{\frac{y}{2}}dxdy\\
			\notag&=\prod_{i=1}^{n}\frac{\sinh\frac{L_i}{2}}{\frac{L_i}{2}}V_{g,n-1}O\left(nLe^{\frac{L}{2}}\right),\\
			\label{B5}
			I_3&\leqslant\frac{1}{3!}\int_{0\leqslant x+y+z\leqslant L}xyzV_{g-2,n+3}\prod_{i=1}^{n}\frac{\sinh\frac{L_i}{2}}{\frac{L_i}{2}}\frac{\sinh\frac{x}{2}}{\frac{x}{2}}\frac{\sinh\frac{y}{2}}{\frac{y}{2}}\frac{\sinh\frac{z}{2}}{\frac{z}{2}}dxdydz\\
			\notag&=\prod_{i=1}^{n}\frac{\sinh\frac{L_i}{2}}{\frac{L_i}{2}}V_{g,n-1}O\left(L^2e^{\frac{L}{2}}\right).
		\end{align}
		
		Note that \cite[Lemma 3.2, Theorem 3.5]{mirzakhani2013growth} implies
		\begin{align}
			\notag V_{g-2,n+3}&\approx V_{g-1.n+1}\approx V_{g,n-1},\\
			\notag V_{g,n-1}&=O\left(\frac{1}{g}\right)V_{g,n}.
		\end{align}
		Hence, combining with (\ref{B2})--(\ref{B5}), it follows that
		\begin{align}
			\label{A2}	&\int_{\mathcal{M}_{g,n}(\mathbf{L})}N_{0,3}(X,L)dX\\
			\notag	=& \prod_{i=1}^{n}\frac{\sinh\frac{L_i}{2}}{\frac{L_i}{2}}\left(O\left(\frac{n^2+nL+L^2}{g}e^{\frac{L}{2}}\right)\right)V_{g,n}\\
			\notag	=& \prod_{i=1}^{n}\frac{\sinh\frac{L_i}{2}}{\frac{L_i}{2}}\left(O\left(\frac{\widetilde{L}^2}{g}e^{\frac{L}{2}}\right)\right)V_{g,n},
		\end{align}
		where $\widetilde{L}=\max\{L,n\}$.
		\\
		
		$(\romannumeral2)$ If $|\chi(S_{(g_0,k+n_0)})|=m$, where $2\leqslant m\leqslant P$. It follows from \cite[Theorem 1.1]{mirzakhani2013growth} that $V_{g_0,k+n_0}(x_1,\dots,x_k,L_{i_1},\dots,L_{i_{n_0}})$
		is a polynomial of degree $6g_0-6+2k+2n_0$, where the coefficients are bounded by some constant $c(m)$ depending only on $m$. Hence, for any  $\sum\limits_{j=1}^{n_0}L_{i_j}\in [0,L]$ with $L>1 $, we have
		\[V_{g_0,k+n_0}(L_{i_1},\dots,L_{i_{n_0}})\leqslant c(m)(L^{6g_0-6+2k+2n_0}).\]
		
		Therefore, we obtain
		\begin{align}
			\notag	&\int_{\mathcal{M}_{g,n}(\mathbf{L})}N_{g_0,k,n_0}(X,L)\\
			\notag	\leqslant&\sum\limits_{i_1<\dots<i_{n_0}}\frac{1}{k!}\int_{0\leqslant\sum\limits_{i=1}^{k}x_i\leqslant L}\prod_{i=1}^{k}x_iV_{g_0,k+n_0}(x_1,\dots,x_k,L_{i_1},\dots,L_{i_{n_0}})\\
			\notag	&\times V_{g-g_0-k+1,k+n-n_0}(x_1,\dots,x_k,L_1,\dots,\widehat{L_{i_1}},\dots,\widehat{L_{i_{n_0}}},\dots,L_n)dx_1\dots dx_k\\
			\notag	\leqslant&\sum_{i_1<\dots<i_{n_0}}\frac{1}{k!}\int_{0\leqslant\sum\limits_{i=1}^{k}x_i\leqslant L}c(m)L^{6g_0-6+2k+2n_0}\prod_{i=1}^{k}x_i\frac{\sinh\frac{x_i}{2}}{\frac{x_i}{2}}V_{g-g_0-k+1,k+n-n_0}\\
			\notag	&\times\prod_{1\leqslant i\leqslant n,i\neq i_1,\dots,i_{n_0}}\frac{\sinh\frac{L_i}{2}}{\frac{L_i}{2}}dx_1\dots dx_k\\
			\notag	\leqslant&\prod_{i=1}^{n}\frac{\sinh\frac{L_i}{2}}{\frac{L_i}{2}}\frac{c(m)L^{6g_0-6+2k+2n_0}}{k!}\binom{n}{n_0}V_{g-g_0-k+1,k+n-n_0}\\
			\notag	&\times \int_{0\leqslant \sum\limits_{i=1}^{k}x_i\leqslant L}\prod_{i=1}^{k}x_i\frac{\sinh\frac{x_i}{2}}{\frac{x_i}{2}}dx_1\dots dx_k.
		\end{align}
		
		Notice that
		\begin{align*}
			\binom{n}{n_0}&\leqslant n^{n_0},\\
			\prod_{i=1}^{k}x_i\frac{\sinh\frac{x_i}{2}}{\frac{x_i}{2}}&=\prod_{i=1}^{k}2\sinh\frac{x_i}{2}
			\leqslant e^{\frac{\sum_{i=1}^{k}x_i}{2}}\leqslant e^{\frac{L}{2}},
		\end{align*}
		and the integral formula
		\begin{equation*}
			\int_{0\leqslant\sum\limits_{i=1}^{k}L_i\leqslant L}dx_1\dots dx_k=\frac{L^k}{k!},
		\end{equation*}
		we combine these properties and get the following estimate
		
		\begin{align}
			\label{A3}	&\int_{\mathcal{M}_{g,n}(\mathbf{L})}N_{g_0,k,n_0}(X,L)\\
			\notag \leqslant&c(m)\prod_{i=1}^{n}\frac{\sinh\frac{L_i}{2}}{\frac{L_i}{2}}L^{6g_0-6+3k+2n_0}e^{\frac{L}{2}}n^{n_0}V_{g-g_0-k+1,k+n-n_0}\\
			\notag \leqslant&c(m)\prod_{i=1}^{n}\frac{\sinh\frac{L_i}{2}}{\frac{L_i}{2}}L^{3m-n_0}e^{\frac{L}{2}}n^{n_0}V_{g,n-(2g_0-2+k+n_0)}\\
			\notag \leqslant&c(m)\prod_{i=1}^{n}\frac{\sinh\frac{L_i}{2}}{\frac{L_i}{2}}e^{\frac{L}{2}}\frac{L^{3m-n_0}n^{n_0}}{g^m}V_{g,n}\\
			\notag \leqslant&\prod_{i=1}^{n}c(m)\frac{\sinh\frac{L_i}{2}}{\frac{L_i}{2}}e^{\frac{L}{2}}\frac{\widetilde{L}^{3m}}{g^m}V_{g,n}.
		\end{align}
		
		For fixed $m$ and $n_0$, there exists finitely many $(g_0,k)$ satisfy  $2g_0-2+k+n_0=m$. Hence, we have
		\begin{align}
			\notag	&\sum_{(g_0,k,n_0):2g_0-2+k+n_0=m}\int_{\mathcal{M}_{g,n}(\mathbf{L})}N_{g_0, k+n_0}(X,L)dX\\
			\notag	&=\sum_{n_0=0}^{t}\sum_{(g_0,k):2g_0-2+k=m-n_0}\int_{\mathcal{M}_{g,n}(\mathbf{L})}N_{g_0,k,n_0}(X,L)dX\\
			\notag	 &\leqslant\sum_{n_0=0}^{t}\sum_{(g_0,k):2g_0-2+k=m-n_0}c(m)\prod_{i=1}^{n}\frac{\sinh\frac{L_i}{2}}{\frac{L_i}{2}}e^{\frac{L}{2}}\frac{\widetilde{L}^{3m}}{g^m}V_{g,n}\\
			\notag	&\leqslant\sum_{n_0=0}^{t}\left( \frac{(m+2)^2}{2}\right)c(m)\prod_{i=1}^{n}\frac{\sinh\frac{L_i}{2}}{\frac{L_i}{2}}e^{\frac{L}{2}}\frac{\widetilde{L}^{3m}}{g^m}V_{g,n}\\
			\notag	&=C(m)\prod_{i=1}^{n}\frac{\sinh\frac{L_i}{2}}{\frac{L_i}{2}}e^{\frac{L}{2}}O \left(\frac{\widetilde{L}^{3m+1}}{g^m}\right)V_{g,n},
		\end{align}
		where $t=\min\{n,m+2\}$ and $C(m)=\frac{(m+2)^2}{2}c(m)$.\\
		
		$(\romannumeral3)$ If $|\chi(S_{g_0,k+n_0})|=2g_0-2+k+n_0=m\geqslant P+1$, then
		\begin{align}
			\notag	&\int_{\mathcal{M}_{g,n}(\mathbf{L})}N_{g_0,k,n_0}(X,L)dX\\
			\notag	\leqslant&\sum\limits_{i_1<\dots<i_{n_0}}\frac{1}{k!}\int_{0\leqslant \sum\limits_{i=1}^k x_i\leqslant L}\prod_{i=1}^{k}x_iV_{g_0,k+n_0}(x_1,\dots,x_k,L_{i_1},\dots,L_{i_{n_0}})\\
			\notag	&\times V_{g-g_0-k+1,k+n-n_0}(x_1,\dots,x_k,L_1,\dots,\widehat{L_{i_1}},\dots,\widehat{L_{i_{n_0}}},\dots,L_n)dx_1\dots dx_k\\
			\notag \leqslant& \sum\limits_{i_1<\dots<i_{n_0}}\frac{1}{k!}\int_{0\leqslant \sum\limits_{i=1}^k x_i \leqslant L}\prod_{i=1}\left(\frac{\sinh\frac{x_i}{2}}{\frac{x_i}{2}}\right)^2\\
			\notag&\times\prod_{i=1}^{n}\frac{\sinh\frac{L_i}{2}}{\frac{L_i}{2}}\prod_{i=1}^{k}x_iV_{g_0,k+n_0}V_{g-g_0-k+1,k+n-n_0}dx_1\dots dx_k\\
			\notag \leqslant& \sum\limits_{i_1<\dots<i_{n_0}}\frac{1}{k!}V_{g_0,k+n_0}V_{g-g_0-k+1,k+n-n_0}\\
			\notag &\times\prod_{i=1}^{n}\frac{\sinh\frac{L_i}{2}}{\frac{L_i}{2}}\int_{0\leqslant \sum\limits_{i=1}^kx_i\leqslant L}e^{\sum\limits_{i=1}^{k}x_i}\prod_{i=1}^{k}x_idx_1\dots dx_k\\
			\notag\leqslant& \sum\limits_{i_1<\dots<i_{n_0}}\frac{1}{k!}V_{g_0,k+n_0}V_{g-g_0-k+1,k+n-n_0}e^{L}\prod_{i=1}^{n}\frac{\sinh\frac{L_i}{2}}{\frac{L_i}{2}}\int_{0\leqslant \sum\limits_{i=1}^kx_i\leqslant L}\prod_{i=1}^{k}x_idx_1\dots dx_k\\
			\notag =&\frac{1}{k!}\binom{n}{n_0}\frac{L^{2k}}{(2k)!}V_{g_0,k+n_0}V_{g-g_0-k+1,k+n-n_0}e^{L}\prod_{i=1}^{n}\frac{\sinh\frac{L_i}{2}}{\frac{L_i}{2}}.
		\end{align}
		
		Hence, we obtain
		\begin{align}
			\notag	&\sum_{(g_0,k,n_0):|\chi(S_{g_0,k+n_0})|\geqslant P+1}\int_{\mathcal{M}_{g,n}(\mathbf{L})}N_{g_0,k+n_0}(X,L)dX\\
			\notag	&=\sum_{g-1+\left[\frac{n}{2}\right]\geqslant 2g_0-2+k+n_0\geqslant P+1,k\geqslant 1}\int_{\mathcal{M}_{g,n}(\mathbf{L})}N_{g_0,k,n_0}(X,L)dX\\
			\notag	&\leqslant\sum_{g-1+\left[\frac{n}{2}\right]\geqslant 2g_0-2+k+n_0\geqslant P+1,k\geqslant 1}\frac{1}{k!}\binom{n}{n_0}\frac{L^{2k}}{(2k)!}V_{g_0,k+n_0}V_{g-g_0-k+1,k+n-n_0}e^{L}\prod_{i=1}^{n}\frac{\sinh\frac{L_i}{2}}{\frac{L_i}{2}}.
		\end{align}
		
		By \cite[Lemma 3.2]{mirzakhani2013growth}, we can see that $V_{g-1,n+2}\leqslant V_{g,n}$ for any $n\geqslant 2$. Now we find $k'\in\{1,2,3\}$ such that $k+n_0-k'$ is even.
		
		If $n$ is even, then
		\begin{equation}
			\notag V_{g-g_0-k+1,,k+n-n_0}\leqslant V_{g-g_0-k+1+\frac{n}{2}-n_0,k+n_0}\leqslant V_{g-g_0-k+1+\frac{n}{2}+\frac{k+n_0-k'}{2},k'}.
		\end{equation}
		On the other hand, $V_{g_0,k+n_0}\leqslant V_{g_0+\frac{k+n_0-k'}{2},k'}$.
		
		Let $u=\max\left\{0,\left[\frac{P+3-k-n_0}{2}\right]\right\}$,
		then \cite[Corollary 3.7]{mirzakhani2013growth} implies
		\begin{align}
			\notag	&\sum_{g_0=u}^{\frac{g-1+\left[\frac{n}{2}\right]-k-n_0}{2}}V_{g_0,k+n_0}V_{g-g_0-k+1,k+n-n_0}\\
			\notag	&\leqslant\sum_{g_0=u}^{\frac{g-1+\left[\frac{n}{2}\right]-k-n_0}{2}} V_{g_0+\frac{k+n_0-k'}{2},k'}V_{g-g_0-k+1+\frac{n}{2}+\frac{k+n_0-k'}{2}-n_0,k'}\\
			\notag	&\leqslant D\frac{V_{g+\frac{n}{2}}}{g^{2u+k+n_0-2}},
		\end{align}
		where $D$ is a constant independent on $(g,n)$.
		
		By \cite[Theorem 3.5]{mirzakhani2013growth} and taking $n=g^{o(1)}$, we see $\frac{V_{g,n}}{V_{g+\frac{n}{2}}}=1+O\left(\frac{n^2}{4g}\right)$, yielding that
		\begin{equation}
			\notag	\sum_{g_0=u}^{\left[\frac{2g-2+n}{2}\right]}V_{g_0,k+n_0}V_{g-g_0-k+1,k+n-n_0}\leqslant D\frac{V_{g,n}}{g^{2u+k+n_0-2}}.
		\end{equation}
		
		Otherwise, if $n$ is odd, by \cite[Theorem 3.5]{mirzakhani2013growth}, there exists a constant $A_1$ independent on $(g,n)$ such that
		\begin{equation}
			\notag	V_{g-g_0-k+1,k+n-n_0}\leqslant A_1\frac{V_{g-g_0-k+1,k+(n+1)-n_0}}{2g-2g_0-k+n-n_0}\leqslant A\frac{V_{g-g_0-k+1,k+(n+1)-n_0}}{g-1+[\frac{n}{2}]}.
		\end{equation}
		
		Since $n+1$ is even, then we have
		\begin{align}
			\notag	&\sum_{g_0=u}^{\frac{g-1+\left[\frac{n}{2}\right]-k-n_0}{2}}V_{g_0,k+n_0}V_{g-g_0-k+1,k+n-n_0}\\
			\notag	&= \sum_{g_0=u}^{\frac{g-1+\left[\frac{n}{2}\right]-k-n_0}{2}}\frac{A_1}{g-1+[\frac{n}{2}]}V_{g_0,k+n_0}V_{g-g_0-k+1,k+(n+1)-n_0}\\
			\notag	 &\leqslant\sum_{g_0=u}^{\frac{g-1+\left[\frac{n}{2}\right]-k-n_0}{2}}\frac{A_1}{g-1+[\frac{n}{2}]}V_{g_0+\frac{k+n_0-k'}{2},k'}V_{g-g_0+\frac{n+1}{2}-k+1-n_0+\frac{k+n_0-k'}{2},k'}\\
			\notag	&=\frac{A_1D}{g-1+[\frac{n}{2}]}\frac{V_{g+\frac{n+1}{2}}}{g^{2u+k+n_0-2}}.
		\end{align}
		Moreover, take $n=g^{o(1)}$, it follows from \cite[Theorem 3.5]{mirzakhani2013growth} that  \[\frac{V_{g,n+1}}{V_{g+\frac{n+1}{2}}}=1+O\left(\frac{(n+1)^2}{4g}\right).\]
		By \cite[Lemma 3.2]{mirzakhani2013growth}, there exists a constant $A_2$ independent on $(g,n)$ that
		\begin{equation}
			\notag	V_{g,n+1}\leqslant A_2(2g-2+n+1)V_{g,n}.
		\end{equation}
		
		Then we can deduce
		\begin{align}
			\notag&\sum_{g_0=u}^{\frac{g-1+\left[\frac{n}{2}\right]-k-n_0}{2}}V_{g_0,k+n_0}V_{g-g_0-k+1,k+n-n_0}\\
			\notag&= \frac{A_1A_2D(2g-2+n+1)}{g-1+[\frac{n}{2}]}\frac{V_{g,n}}{g^{2u+k+n_0-2}}\left(1+O\left(\frac{(n+1)^2}{4g}\right)\right)\\
			\notag&= 3A_1A_2D\frac{V_{g,n}}{g^{2u+k+n_0-2}}\left(1+O\left(\frac{(n+1)^2}{4g}\right)\right).
		\end{align}
		
		We summarize the above discussions to see that there exists a universal constant $D'$ independent on $g,n$ such that for each $n$,
		\begin{equation}
			\notag	\sum_{g_0=u}^{\frac{g-1+\left[\frac{n}{2}\right]-k-n_0}{2}}V_{g_0,k+n_0}V_{g-g_0-k+1,k+n-n_0}\leqslant D'\frac{V_{g,n}}{g^{2u+k+n_0-2}}.
		\end{equation}
		
		Hence we deduce that
		\begin{align}
			\label{C1}	&\sum_{(g_0,k,n_0):P+1\leqslant|\chi(S_{g_0,k+n_0})|\leqslant g-1+[\frac{n}{2}]}\int_{\mathcal{M}_{g,n}(\mathbf{L})}N_{g_0,m-2g_0+2}(X,L)dX\\
			\notag\leqslant&\sum_{1\leqslant k+n_0\leqslant g-1+[\frac{n}{2}]+2,k\geqslant 1}\frac{1}{k!}\binom{n}{n_0}\frac{L^{2k}}{(2k)!}e^{L}\prod_{i=1}^{n}\frac{\sinh\frac{L_i}{2}}{\frac{L_i}{2}}D'\frac{V_{g,n}}{g^{2u+k+n_0-2}}\\
			\notag\leqslant& \sum_{1\leqslant k+n_0\leqslant P+2,k\geqslant 1}\frac{1}{k!}\binom{n}{n_0}\frac{L^{2k}}{(2k)!}e^{L}\prod_{i=1}^{n}\frac{\sinh\frac{L_i}{2}}{\frac{L_i}{2}}D'\frac{V_{g,n}}{g^{P}}\\
			\notag	&+\sum_{P+3\leqslant k+n_0\leqslant g-1+[\frac{n}{2}]+2,k\geqslant 1}\frac{1}{k!}\binom{n}{n_0}\frac{L^{2k}}{(2k)!}e^{L}\prod_{i=1}^{n}\frac{\sinh\frac{L_i}{2}}{\frac{L_i}{2}}D'\frac{V_{g,n}}{g^{-2+k+n_0}}\\
			\notag =:& J_1+J_2.
		\end{align}
		
		For $J_1$, it can be estimated that
		\begin{align}
			\label{C2}	J_1&\leqslant e^{L}\prod_{i=1}^{n}\frac{\sinh\frac{L_i}{2}}{\frac{L_i}{2}}D'\frac{V_{g,n}}{g^{P}}\sum_{n_0=0}^{P}\binom{n}{n_0}\sum_{k=1}^{P+2-n_0}\frac{1}{k!}\frac{L^{2k}}{(2k)!}\\
			\notag&\leqslant e^{2L}\prod_{i=1}^{n}\frac{\sinh\frac{L_i}{2}}{\frac{L_i}{2}}D'\frac{V_{g,n}}{g^{P}}\sum_{n_0=0}^{P}n^{n_0}\\
			\notag&\leqslant e^{2L}\prod_{i=1}^{n}\frac{\sinh\frac{L_i}{2}}{\frac{L_i}{2}}D'\frac{V_{g,n}}{g^{m+1}}\sum_{n_0=0}^{P}n^{n_0}\\
			\notag	&= e^{2L}\prod_{i=1}^{n}\frac{\sinh\frac{L_i}{2}}{\frac{L_i}{2}} V_{g,n} O\left(\frac{(P+1)n^{P}}{g^{P-1}}\right).
		\end{align}
		
		And for $J_2$,
		\begin{align}
			\notag	J_2&\leqslant e^{2L}V_{g,n}\prod_{i=1}^{n}\frac{\sinh\frac{L_i}{2}}{\frac{L_i}{2}}\sum_{m+3\leqslant k+n_0\leqslant g-1+[\frac{n}{2}]+2,k\geqslant 1}D'\frac{n^{n_0}}{g^{-2+k+n_0}}\\
			\notag	&\leqslant e^{2L}V_{g,n}\prod_{i=1}^{n}\frac{\sinh\frac{L_i}{2}}{\frac{L_i}{2}}\sum_{m+3\leqslant k+n_0\leqslant g-1+[\frac{n}{2}]+2,k\geqslant 1}D'\frac{n^{n_0+k-2+1}}{g^{-2+k+n_0}}.
		\end{align}
		
		Let $t=k+n_0$,
		\begin{align}
			\label{C3}J_2&\leqslant 2 e^{2L}V_{g,n}\prod_{i=1}^{n}\frac{\sinh\frac{L_i}{2}}{\frac{L_i}{2}}( g-1+[\frac{n}{2}]+2)\sum_{P+3\leqslant t,k\geqslant 1}D'\frac{n^{t-1}}{g^{t-2}}\\
			\notag&=2 e^{2L}V_{g,n}\prod_{i=1}^{n}\frac{\sinh\frac{L_i}{2}}{\frac{L_i}{2}}O\left(\frac{n^{P+2}}{g^{P}}\right),
		\end{align}
		where we apply the geometric series formula and take $n=g^{o(1)}$ to obtain the last formula.\\
		
		In the end, we take $P=10$, and combine with the above estimates (\ref{A1}), (\ref{A2})--(\ref{C3}) to deduce that
		
		\begin{align}
			\notag&\operatorname{Vol}(\{X\in\mathcal{M}_{g,n}(\mathbf{L});\mathcal{L}_1(X)\leqslant L\})\\
			\notag\leqslant& \int_{\mathcal{M}_{g,n}(\mathbf{L})}N_{1,1}(X,L)dX+ \int_{\mathcal{M}_{g,n}(\mathbf{L})}N_{0,3}(X,L)dX\\
			\notag&+\sum_{(g_0,k,n_0):2\leqslant 2g_0-2+k+n_0\leqslant 10}\int_{\mathcal{M}_{g,n}(\mathbf{L})}N_{g_0,k,n_0}(X,L)dX\\
			\notag&+\sum_{(g_0,k,n_0):11\leqslant 2g_0-2+k+n_0\leqslant g-1+[\frac{n}{2}]}\int_{\mathcal{M}_{g,n}(\mathbf{L})}N_{g_0,k,n_0}(X,L)dX\\
			\notag\leqslant& \prod_{i=1}^{n}\frac{\sinh \frac{L_i}{2}}{\frac{L_i}{2}}V_{g,n}O\left(\frac{\widetilde{L}^2}{g}e^{\frac{L}{2}}+\frac{\widetilde{L}^2}{g}e^{\frac{L}{2}}+\sum_{m=2}^{10}C(m)e^{\frac{L}{2}}\frac{\widetilde{L}^{3m+1}}{g^m}+e^{2L}\frac{11n^{10}}{g^{9}}+e^{2L}\frac{n^{12}}{g^{10}}\right).
		\end{align}
		
		Then by \cite[Lemma 22]{nie2023large} and take $n=g^{o(1)},\sum\limits_{i=1}^{n}L_i=g^{o(1)}$, it follows that
		\begin{equation}
			\notag V_{g,n}(\mathbf{L})=(1-O(g^{-1+o(1)}))\prod_{i=1}^{k}\frac{\sinh \frac{L_i}{2}}{\frac{L_i}{2}}V_{g,n}.
		\end{equation}
		Hence we have
		\begin{equation}
			\label{Eq5.1}	\frac{\operatorname{Vol}(\{X\in\mathcal{M}_{g,n}(\mathbf{L});\mathcal{L}_1(X)\leqslant L(g)\})}{V_{g,n}(\mathbf{L})}= O\left(\frac{\widetilde{L}^2}{g}e^{\frac{L}{2}}+Ce^{\frac{L}{2}}\frac{\widetilde{L}^{7}}{g^2}+e^{2L}\frac{11n^{10}}{g^{9}}+e^{2L}\frac{n^{12}}{g^{10}}\right),
		\end{equation}
		where $C=\max\limits_{2\leqslant m\leqslant 10}C(m)$.\\
		
		Therefore, we take $n=c\log g$ and $L(g)=(2-\varepsilon)\log g$ for some $\varepsilon,c>0$ to obtain
		\begin{equation*}
			\lim\limits_{g\rightarrow\infty}\frac{\operatorname{Vol}(\{X\in\mathcal{M}_{g,n}(\mathbf{L});\mathcal{L}_1(X)\leqslant (2-\varepsilon)\log g\})}{V_{g,n}(\mathbf{L})}=0,
		\end{equation*}
		which we complete the proof of Theorem \ref{Thm5.1}.
	\end{proof}
	
	Furthermore, taking $n\leqslant c\log g$ and $L= A\log\log g$ for some fixed constants $c$ and $A$, we get the following corollary, which will be used to prove Theorem \ref{Thm1.1},
	\begin{cor}\label{C4.2}
		Let $n\leqslant c\log g$ and $L_i\leqslant L=A\log\log g$ for $i=1,\dots,n$, we have
		\begin{equation}
			\label{T3}		\frac{\operatorname{Vol}(\mathcal{N}_{g,n}^L(\mathbf{L}))}{V_{g,n}(\mathbf{L})}=1-O(g^{-1+o(1)}).
		\end{equation}
	\end{cor}
	\begin{proof}
		We only need to prove that $\left(\mathcal{N}_{g,n}^L(\mathbf{L})\right)^c\subseteq E_{g,n}^{5L}(\mathbf{L})$.
		
		If $X\in\mathcal{M}_{g,n}(\mathbf{L})$ but $X\not\in\mathcal{N}_{g,n}^L(\mathbf{L})$, then $w(X)<\frac{1}{4}L$. There are two cases:\\
		
		\textbf{Case 1}: There exists a closed geodesic or a boundary geodesic $\gamma$ with length less than $L$ and $w_{\gamma}=w(X)$. Then there exists a point $P_0$ on the boundary $\partial \mathcal{C}(\gamma,w(X))-\gamma$ and two distinct points $P_1$ and $P_2$ on $\gamma$ such that the distances $d(P_0,P_1)$ and $d(P_0,P_2)$ both equal to $w(X)$. $P_1$ and $P_2$ divide $\gamma$ into  two sub-geodesics $\gamma_1$ and $\gamma_2$. We denote the geodesic arc from $P_2$ to $P_1$ that passes through $P_0$ as $\gamma_0$. The geodesic arcs $\gamma_0$, $\gamma_1$, and $\gamma_2$ are then given appropriate orientation(see Figure \ref{fig1} for an illustration).
		\begin{figure}[b]
			\centering
			\includegraphics[width=10cm]{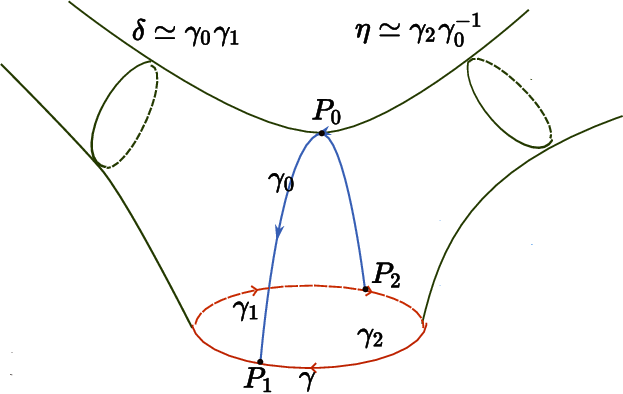}
			\caption{$w_{\gamma}<\frac{1}{4}L$}
			\label{fig1}
		\end{figure}
		
		Consider the closed curves $\gamma_1\gamma_0$ and $\gamma_0^{-1}\gamma_2$, which are homotopic to closed geodesics $\delta$ and $\eta$, respectively. Together with $\gamma$, these geodesics partition the surface $X$ into a pair of pants and several other components. We can choose a subset of the three geodesics to partition $X$ into two parts such that the total length of the selected geodesics is no greater than $2\ell(\gamma)+2\ell(\gamma_0)$, and this value is in turn no greater than $3L$.\\
		
		\textbf{Case 2}: Two collars or half-collars $\partial\mathcal{C}(\gamma_1,w(X)) \cap \partial \mathcal{C}(\gamma_2,w(X))\neq \varnothing $, where $\gamma_2$ may be $\gamma_1^{-1}$ with an approximate orientation. This implies the existence of a point $P_0 \in \partial \mathcal{C}(\gamma_1,w(X))\cap\partial \mathcal{C}(\gamma_2,w(X))$, along with points $P_1$ on $\gamma_1$ and $P_2$ on $\gamma_2$ such that $d(P_0,P_1)=d(P_0,P_2)=w(X)<\frac{1}{4}L$. We denote by $\gamma_0$ the geodesic arc from $P_1$ to $P_2$ that passes through $P_0$. The geodesic arcs $\gamma_0$, $\gamma_1$, and $\gamma_2$ are each oriented appropriately. See Figure \ref{fig2} for an illustration.
		\begin{figure}
			\centering
			\includegraphics[width=6cm]{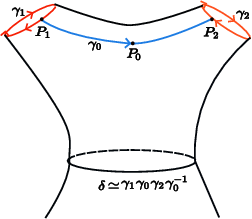}
			\caption{$\frac{1}{2}\operatorname{dist}(\gamma_1,\gamma_2)<\frac{1}{4}L$}
			\label{fig2}
		\end{figure}
		Consider the closed curve $\gamma_1\gamma_0\gamma_2\gamma_0^{-1}$, which is homotopic to the closed geodesic $\delta$. The geodesics $\delta$, $\gamma_1$, and $\gamma_2$ partition $X$ into a pair of pants and several other components. We can choose a subset of these three geodesics to partition $X$ into two parts, with the total length of the selected geodesics bounded by $2\ell(\gamma_0)+2\ell(\gamma_1)+2\ell(\gamma_2)$, which is at most $5L$.\\
		
		In both cases, we observe that $\mathcal{L}_1(X)\leqslant 5L$, and hence
		\begin{equation}
			\notag	(\mathcal{N}_{g,n}^L(\mathbf{L}))^c\subseteq E_{g,n}^{5L}(\mathbf{L}).
		\end{equation}
		
		Now we use (\ref{Eq5.1}) and $L=A\log \log g$, we see that
		\begin{equation}
			\notag		\frac{\operatorname{Vol}(E_{g,n}^{5L}(\mathbf{L}))}{V_{g,n}(\mathbf{L})}= O(g^{-1+o(1)}),
		\end{equation}
		and obtain the corollary.
	\end{proof}
	
	\section{\noindent{{\bf Proof of Theorem \ref{Thm1.1} and Theorem \ref{Thm1.11}}}}
	\label{S6}
	
	In this section, we absorb the above conclusions to prove Theorem \ref{Thm1.1}.
	Our approach utilizes Mirzakhani’s integration formula \cite{mirzakhani2007simple} to analyze the distribution of short geodesics. However, the possibility of intersection arises for geodesics with length greater than $2\operatorname{arcsinh}1$. Consequently, we restrict to the subset  $\mathcal{N}_g^L$ of Teichm\"uller space $\mathcal{M}_g$ as defined in Sect. \ref{S3} to prove Theorem \ref{Thm1.11}.

	\begin{proof}[Proof of Theorem \ref{Thm1.11}]
		First observe that for a hyperbolic surface $X \in \mathcal{N}_g^L$, any two of the short closed geodesics are disjoint, and the number of the primitive short closed geodesics is at most $3g - 3$.
		
		Let $S_k$ represent the volume of the subset of $\mathcal{N}_g^L$ comprising surfaces happen with precisely $k$ primitive short closed geodesics are non-separating. Given that there exists at most $g$ disjoint primitive non-separating closed geodesics, it follows that
		\begin{equation}
			\notag   	\operatorname{Vol}(\mathcal{N}_g^{< L})=\int_{\mathcal{M}_g}1_{\mathcal{N}_g^L}dX\geqslant\sum_{k=1}^{g} S_k.
		\end{equation}
		
		Let $G_k$ be the integral of the number of sets $S$ of $k$ disjoint unoriented short geodesics, and all such geodesics are not-separating. We can see that
		\begin{equation}
			\notag    	G_k=\sum_{r=k}^{g} \binom{r}{k}S_k=\sum_{r=1}^{g}\binom{r}{k}S_k.
		\end{equation}
		Here, $\binom{r}{k}=0$ when $r<k$. By the inclusion exclusion principle
		\begin{equation}
			\notag	\sum_{k=1}^{n}(-1)^{k+1}\binom{r}{k}=1-(-1)^n\binom{r-1}{n},
		\end{equation}
		it follows that
		\begin{align}
			\notag	\sum_{k=1}^{n}(-1)^{k+1}G_k&=\sum_{k=1}^{n}\sum_{r=1}^{g}(-1)^{k+1}\binom{r}{k}S_k\\
			\notag	&=\sum_{r=1}^{g}\sum_{k=1}^{n}(-1)^{k+1}\binom{r}{k}S_k\\
			\notag	&=\sum_{r=1}^{g}\left(1-(-1)^n\binom{r-1}{n}\right)S_k.
		\end{align}
		
		If we take $n=g$, then
		\begin{equation}
			\notag	\sum_{k=1}^{g}(-1)^{k+1}G_k=\sum_{r=1}^{g}S_k\leqslant\operatorname{Vol}(\mathcal{N}^L_g),
		\end{equation}
		since $\binom{r-1}{g}=0$ for all $r\in\{1,2,\dots,g\}$.
		
		In particular, if $n$ is an even number, then
		\begin{equation}\label{Eq6}
			\sum_{k=1}^{n}(-1)^{k+1}G_k=\sum_{r=1}^{g}\left(1-\binom{r-1}{n}\right)S_k\leqslant\sum_{k=1}^{g}S_k\leqslant\operatorname{Vol}(\mathcal{N}^L_g).
		\end{equation}
		
		For fixed positive constants $\frac{1}{2}<c_1<c_2$ and $g$ large enough, we can choose $n$ be an even number which satisfies $c_1\log g\leqslant n\leqslant c_2\log g$  and combine with (\ref{T2}) to calculate $G_k$ in $\mathcal{N}_g^L$, where
		\begin{align}
			\notag	G_k&=\int_{\mathcal{N}_g^L} 1_{\{x_i<L\}}dX\\
			\notag	&=\frac{1}{2^kk!}\int_{0}^{L}\dots\int_{0}^{L}x_1\dots x_k\operatorname{Vol}\left(\mathcal{N}_{g-k,2k}^L(x_1,x_1,\dots,x_k,x_k)\right)dx_1\dots dx_k.
		\end{align}
		Suppose $k\leqslant n=g^{o(1)}$ and $2\sum\limits_{i=1}^{k}x_i\leqslant nL(g)=g^{o(1)}$, then we deduce that
		\begin{align}
			\notag	\frac{\operatorname{Vol}\left(\mathcal{N}_{g-k,2k}^L(x_1,x_1,\dots,x_k,x_k)\right)}{V_{g-k,2k}(x_1,x_1,\dots,x_k,x_k)}&=1-O(g^{-1+o(1)})\\
			\notag	 \frac{V_{g-2k,2k}(x_1,x_1,\dots,x_k,x_k)}{V_{g-k,2k}}&=\left(1+O(g^{-1+o(1)})\right)\prod_{i=1}^{k}\left(\frac{\sinh\frac{x_i}{2}}{\frac{x_i}{2}}\right)^2\\
			\notag	\frac{V_{g-k,2k}}{V_g}&=1+O(g^{-1+o(1)})
		\end{align}
		from (\ref{T3}), \cite[Lemma 22]{nie2023large}, and \cite[Theorem 3.5]{mirzakhani2013growth}. Hence, it follows that
		\begin{align}
			\notag	G_k&=\frac{1-O(g^{-1+o(1)})}{2^kk!}\int_{0}^{L}\dots\int_{0}^{L}x_1\dots x_k\left(V_{g-k,2k}(x_1,x_1,\dots,x_k,x_k)\right)dx_1\dots dx_k\\
			\notag	&=\frac{1-O(g^{-1+o(1)})}{2^kk!}\int_{0}^{L}\dots\int_{0}^{L}x_1\dots x_k\prod_{i=1}^{k}\left(\frac{\sinh\frac{x_i}{2}}{\frac{x_i}{2}}\right)^2V_{g-k,2k}dx_1\dots dx_k\\
			\notag	&=\frac{I^k}{k!}V_g\left(1+O(g^{-1+o(1)})\right),
		\end{align}
		where
		\begin{equation}
			\notag	I=\frac{1}{2}\int_{0}^{L}x\left(\frac{\sinh\frac{x}{2}}{\frac{x}{2}}\right)^2dx.
		\end{equation}
		
		Thanks to Taylor's theorem, we can see that
		\begin{align}
			\label{E1}	\frac{1}{V_g}\sum\limits_{k=1}^{n}(-1)^{k+1}G_k &= \sum_{k=1}^{n}(-1)^{k+1}\frac{I^k}{k!}\left(1+O(g^{-1+o(1)})\right)\\
			\notag	&=(1-e^{-I}-e^{-\xi}\frac{I^{n+1}}{(n+1)!})-O(e^Ig^{-1+o(1)}) \\
			\notag	&= \left(1-e^{-I}-O\left(\frac{I^{n+1}}{(n+1)!}\right)\right)-O(e^Ig^{-1+o(1)}).
		\end{align}
		By Stirling's approximation
		\begin{equation}
			\notag	\frac{I^{n+1}}{(n+1)!}=\frac{I^{n+1}}{\sqrt{2\pi (n+1)}\left(\frac{n+1}{e}\right)^{n+1} e^{\theta}}\leqslant\frac{I^{n+1}}{\left(\frac{c_1\log g+1}{e}\right)^{n+1}},
		\end{equation}
		where $\theta=O(n^{-1})$. And for $I$
		\begin{align}
			\notag	 I&=\frac{1}{2}\int_{0}^{1}x\left(\frac{\sinh\frac{x}{2}}{\frac{x}{2}}\right)^2dx+\frac{1}{2}\int_{1}^{L}x\left(\frac{\sinh\frac{x}{2}}{\frac{x}{2}}\right)^2dx
			\notag	&=C+\frac{1}{2}\int_{1}^{L}x\left(\frac{\sinh\frac{x}{2}}{\frac{x}{2}}\right)^2dx
		\end{align}
		where $C=\frac{1}{2}\int_{0}^{1}x\left(\frac{\sinh\frac{x}{2}}{\frac{x}{2}}\right)^2dx$. In particular,
		
		\begin{equation}
			\notag	 \int_{1}^{L}\frac{e^x-2}{2x}dx\leqslant\frac{1}{2}\int_{1}^{L}x\left(\frac{\sinh\frac{x}{2}}{\frac{x}{2}}\right)^2dx\leqslant\int_{1}^{L}\frac{e^x}{2x}dx\leqslant \frac{1}{2}e^L.
		\end{equation}
		Now we denote
		\begin{equation}
			\notag T(L):=\int_1^L \frac{e^x}{2x}dx,
		\end{equation}
		then
		\begin{equation}
			\notag T(L)-\log L +C \leqslant I \leqslant T(L) +C \leqslant \frac{1}{2}e^L+C.
		\end{equation}
		
		Conbining with (\ref{E1}), we get
		\begin{equation}
			\label{Eq6.1}	\frac{1}{V_g}\sum\limits_{k=1}^{n}(-1)^{k+1}G_k =\left(1-e^{-T(L)+\log L -C}-O\left(\frac{eT(L)+Ce}{c_1\log g+1}\right)^{n+1}\right)+O\left(g^{-1+o(1)}e^{T(L)+C}\right).
		\end{equation}
		If $L\leqslant \log (2-\delta)\log g$ and $c_1>(2+C)e^2$, we get
		\begin{equation}
			\notag \left(\frac{eT(L)+Ce}{c_1\log g+1}\right)^{n+1}\leqslant\left(\frac{e^{L+1}+Ce}{c_1\log g+1}\right)^{n+1}\leqslant e^{-n-1}\leqslant g^{-c_1}.
		\end{equation}
		Then, we conclude that
		\begin{equation}
			\notag	\frac{1}{V_g}\sum\limits_{k=1}^{n}(-1)^{k+1}G_k = \left(1-O(Le^{-T(L)}-g^{-c_1})\right)+O\left(g^{-1+o(1)}e^{T(L)}\right).
		\end{equation}
		Notice that $T(L)$ is monotonically increasing for $L\geqslant 2$, thus the inverse function $T^{-1}$ exists and is also monotonically increasing, which yields
		\begin{equation}
			\notag	\frac{e^L-e}{2L}\leqslant T(L) \leqslant \frac{1}{2}e^L.
		\end{equation}
		
		In particular,
		\begin{equation}
			\notag T(L)\leqslant (1-\frac{\delta}{2})\log g,
		\end{equation}
		when $L(g)\leqslant \log \left((2-\delta)\log g\right)$. Therefore, we have
		\begin{equation}
			\notag	\frac{1}{V_g}\sum\limits_{k=1}^{n}(-1)^{k+1}G_k =\left(1-O(Le^{-\frac{e^L-e}{2L}}-g^{-c_1})\right)+O\left(g^{-\frac{\delta}{2}+o(1)}\right).
		\end{equation}
		If $\lim\limits_{g\rightarrow\infty}L(g)=\infty$, it follows that
		\begin{equation}
			\notag	\lim\limits_{g\rightarrow\infty}Le^{-\frac{e^L-e}{2L}}=0,\quad \lim\limits_{g\rightarrow\infty}g^{-\frac{\delta}{2}+o(1)}=0.
		\end{equation}
		Combined with \eqref{Eq6}, we complete the first conclusion that
		\begin{equation}
			\notag \lim\limits_{g\rightarrow\infty} \frac{\operatorname{Vol}(\mathcal{N}_g^L)}{V_g}=1.
		\end{equation}
		
		Furthermore, if we take $L_0(g)=T^{-1}(\frac{1}{2}\log g)$, it follows from (\ref{Eq6.1}) that
		\begin{equation}
			\notag	\frac{\operatorname{Vol}(\mathcal{N}_g^{L_0(g)})}{V_g}= 1-O(g^{-\frac{1}{2}+o(1)}),
		\end{equation}
		and thus we complete the proof.
	\end{proof}
	\begin{rmk}
		In fact, if we take $L(g)$ more precisely to satisfy that
		\begin{equation}
			\lim\limits_{g\rightarrow\infty}L(g)=\infty,\quad L(g)\leqslant T^{-1}\left((1-\delta)\log g\right),
		\end{equation}
		then we can also obtain
		\begin{equation}
			\notag \lim\limits_{g\rightarrow\infty} \frac{\operatorname{Vol}(\mathcal{N}_g^L)}{V_g}=1.
		\end{equation}
	\end{rmk}
	
	Now we use Theorem \ref{Thm1.11} to prove Theorem \ref{Thm1.1}.
	
	\begin{proof}[Proof of Theorem  \ref{Thm1.1}]
		
		If $L(g)\geqslant  (1+\varepsilon)\log\log g$, we can see that
		\begin{equation}
			\notag	T(L)\geqslant \frac{(\log g)^{1+\varepsilon}-e}{2(1+\varepsilon)\log\log g}>\frac{1}{2}\log g,
		\end{equation}
		when $g$ is large enough. Since $T(L)$ is monotonically increasing with respect to $L$, we can see that $L(g)\geqslant L_0(g)$. By the definition of $\mathcal{N}_g^L$, it follows that
		\begin{equation}
			\notag \mathcal{N}_g^{L_0}\subseteq \mathcal{M}_g^{<L}.
		\end{equation}
		
		Thus
		\begin{equation}
			\notag \operatorname{Prob}_{\mathrm{WP}}^g(X\in\mathcal{M}_g ; \ell_{sys}(X)< L(g))\geqslant \frac{\operatorname{Vol}(\mathcal{N}_g^{L_0})}{V_g}=1-O(g^{-\frac{1}{2}+o(1)}),
		\end{equation}
		and hence we finish the proof of the main theorem.
	\end{proof}
	
	\begin{rmk}
		If we take $L(g)=l$ be a constant and estimate $I=\frac{1}{2}\int_{0}^{l}x\left(\frac{\sinh\frac{x}{2}}{\frac{x}{2}}\right)^2dx$ more accurately as
		\begin{equation}
			\notag	 I=\frac{1}{2}\int_{0}^{l}x\left(\frac{\sinh\frac{x}{2}}{\frac{x}{2}}\right)^2dx=\int_{0}^{l}\frac{e^x+e^{-x}-2}{2x}dx\geqslant\int_{0}^{l}\frac{x}{2}dx=\frac{l^2}{4}.
		\end{equation}
		It follows from (\ref{E1}) that
		\begin{equation}
			\notag \lim\limits_{g\rightarrow\infty} \frac{\operatorname{Vol}(\mathcal{N}_g^l)}{V_g}\geqslant 1-e^{-\frac{l^2}{4}} .
		\end{equation}
		In particular, let $l=\varepsilon>0$ small enough, $\mathcal{N}_g^{\varepsilon}$ is precisely $\mathcal{M}_g^{<\varepsilon}$. In the meantime,
		\begin{equation}
			\notag		1-e^{-I}\approx 1-e^{-\frac{l^2}{4}}\approx \frac{\varepsilon^2}{4}
		\end{equation}
		is the volume of the subset consists of the surfaces with all closed geodesics whose length less than $\varepsilon$ are simple and non-separating. On the other hand, the volume of the subsets that the surfaces have a closed separating geodesic whose length less than $\varepsilon$ is approximate to 0 when $g \to \infty$ \cite{nie2023large}.
		Hence,
		\[
		\lim\limits_{g\rightarrow\infty}\frac{\operatorname{Vol}(\mathcal{M}_g^{<\varepsilon})}{V_g}\approx\frac{\varepsilon^2}{4},
		\]
		which conincides with the asymptotic behavior for the volume of $\mathcal{M}_g^{<\varepsilon}$
		obtained in \cite[Proposition 3.3]{lipnowski2024towards} and \cite[Theorem 4.1]{mirzakhani2013growth}.
	\end{rmk}
	
	\bibliography{main}
	\bibliographystyle{plain}{}
\end{document}